\documentclass[12pt]{article}
\usepackage{amsgen,amsmath,amstext,amsbsy,amsopn,amsfonts,amssymb}

\newtheorem{theorem}{Theorem}
\newtheorem{lemma}[theorem]{Lemma}

\newtheorem{corollary}[theorem]{Corollary}
\newtheorem{proposition}[theorem]{Proposition}
\newtheorem{obs}[theorem]{Observation} \newtheorem{defi}[theorem]{Definition}

\newenvironment{definition}{\begin{defi}\rm}{\end{defi}}
\newtheorem{exa}[theorem]{Example}
\newenvironment{example}{\begin{exa}\rm}{\end{exa}}
\newtheorem{rem}[theorem]{Remark}
\newenvironment{remark}{\begin{rem}\rm}{\end{rem}}
\newtheorem{rems}[theorem]{Remarks}

\newtheorem{ack}[theorem]{Acknowlegment}

\def\bsq{\blacksquare\medskip}
\def\n{\noindent}
\def\H{\mathcal H}
\def\L{\mathcal L}
\def\K{\mathcal K}

\def\B{\mathcal B}

\def\supp{{\rm supp}}

\def\ZZ{{\mathbf Z}}
\def\CCC{{\mathbf C}}
\def\RRR{{\mathbf R}}
\def\QQ{\mathbf Q}
\def\RR+{{\mathbf R}^*}
\def\TT{\mathbf T}

\def\PP{\mathbf P}

\def\Q_p{{\mathbf Q}_p}

\def\ind{{\rm Ind}}

\def\Tr{{\rm Trace}}

\def\eps{\varepsilon}
\def\Ga{\Gamma}
\def\ga{\gamma}
\def\La{\Lambda}
\def\la{\lambda}

\def\vfi{\varphi}

\def\tous{\qquad\text{for all}\quad}
\def\bs{\backslash}
\def\nil{\La\backslash N}
\def\tor{\La [N,N]\backslash N}
\def\Aut{{\rm Aut}}
\def\Aff{{\rm Aff}}
\def\Affnil{{\rm Aff}(\nil)}
\def\Autnil{\Aut (\nil)}
\def\AffT{{\rm Aff}(T)}
\def\AutT{\Aut (T)}
\def\paut{p_{\rm a}}
\def\ind{{\rm Ind}}
\def\Ker{{\rm Ker}}
\def\Ad{{\rm Ad}}

\def\Tr{{\rm Tr}}
\def\hN{\widehat N}

\def\Utor{U_{\rm tor}}
\def\ZC{{\rm Zc}}
\def\supp{{\rm supp}}

\def\tout{\qquad\text{for all}\quad}

\begin{document}

\title{On the spectral theory of groups of affine transformations of compact nilmanifolds}
\author{Bachir Bekka and Yves Guivarc'h}

\date{\today}

\maketitle

\begin{abstract}
Let $N$ be a connected and simply connected nilpotent Lie group, $\La$  a lattice in $N$,
and $\nil$ the corresponding nilmanifold. 
Let $\Affnil$ be the group of affine transformations of $\nil.$

We   characterize the countable subgroups $H$ 
 of  $\Affnil$ for which  the action of $H$ on $\nil$  has a spectral gap,
 that is, such that  the associated
unitary representation $U^0$ of  $H$ on the space of functions from $L^2(\nil)$ with zero mean
does  not weakly contain the trivial representation.
Denote by $T$ the  maximal  torus factor  associated to  $\nil$. 
We show that the action of $H$  on $\nil$   has  a spectral gap
 if and only if there exists  no proper  $H$-invariant  subtorus  $S$ of $T$ such that the projection of  $H$ on $\Aut (T/S)$  has an abelian subgroup of finite index.
 
We first  establish the result  in the case where
$\nil$ is a  torus. In the case of a general nilmanifold,
we study  the asymptotic behaviour  of matrix coefficients  of 
 $U^0$ using decay properties of  metaplectic representations  of symplectic groups. 
The result shows that the existence of a spectral gap
for subgroups of $\Affnil$ is equivalent to strong ergodicity
in the sense  of K.~Schmidt.
Moreover, we   show that  the action of $H$   on $\nil$ 
is ergodic (or strongly mixing) if and only if the corresponding action of  $H$
on $T$ is ergodic (or strongly mixing).

\end{abstract}

\section{Introduction}
\label{S0}

Let $H$ be a countable group
 acting measurably 
on a probability space $(X,\nu)$  by measure preserving transformations.
Let $U: h\mapsto U(h)$ denote the corresponding Koopman   representation
of $H$ on $L^2(X,\nu)$. We say that 
the action of $H$ on $X$ has a spectral gap
if the restriction $U^0$  of  $U$ to the $H$-invariant subspace
$$
L^2_0(X,\nu)=\{\xi\in L^2(X,\nu) \ :\ \int_X \xi (x) d\nu (x)=0\}
$$
does not have almost invariant vectors,
that is, there is no sequence of unit vectors
 $\xi_n$  in  $ L^2_0(X,\nu)$  such
 that $\lim_n\Vert U^0(h)\xi_n-\xi_n\Vert=0$ for all
 $h\in H.$
A useful  equivalent condition for the existence of a spectral gap
is as follows. Let $\mu$ be a probability
measure on $H$ such that  the support of $\mu$
generates $H$.
Let $U^0(\mu)$ be the convolution operator  defined
on $L^2_0(X,\nu)$ by 
$$
U^0(\mu)\xi =\sum_{h\in H} \mu(h) U^0(h) \xi , \qquad \xi\in L^2_0(X,\nu).
$$
Observe that  we have $\Vert U^0(\mu) \Vert \leq 1$ 
and hence  $r(U^0(\mu)) \leq 1$  for the  spectral 
radius $r(U^0(\mu))$ of  $U^0(\mu)$.
Assume that  $\mu$ is aperiodic,
(that is, if  $\supp (\mu)$ is not contained in the coset of a  proper subgroup of $H$). Then the action of $H$ on $X$ has a spectral gap
if and only if $r(U^0(\mu))<1$ and this is 
 equivalent to  $\Vert U^0(\mu) \Vert<1$.

Ergodic theoretic applications of the existence of a spectral gap (or of the stable spectral gap; see below for the
definition) to random walks  
(such as the rate of $L^2$-convergence in the
random  ergodic theorem, pointwise ergodic theorem, analogues of the law of large numbers and of the central limit theorem, etc)
 are given in  \cite{CoGu2}, \cite{CoLe}, \cite{FuSh}, \cite{GoNe} and  \cite{Yves}.  Another application
of the spectral gap property is the uniqueness of $\nu$ as $H$-invariant mean
on $L^\infty(X,\nu);$ for this as well as for further
applications, see \cite{BHV}, \cite{Lubotzky}, \cite{Popa}, \cite{Sar}.

Recall that a factor $(Y,m,H)$ of  the system $(X, \nu,H)$ is a 
probability space $(Y,m)$ equipped with an $H$-action 
by measure preserving transformations
together with a   $H$-equivariant mesurable  mapping $\Phi: X\to Y$
with $\Phi_*(\nu)=m.$   Observe that $L^2(Y,m)$
can be identified with a $H$-invariant closed subspace of 
$L^2(X,\nu).$

By a result proved in \cite[Theorem 2.4]{JuRo},
 no  action of  a countable amenable group  by measure preserving transformations on a 
 non-atomic probability space  has a spectral gap.
As a consequence,  if there exists
a non-atomic factor $(Y,m, H)$  of  the system $(X, \nu,H)$
such that $H$ acts as an amenable group on $Y,$
then  the action of $H$ on $X$ has no spectral gap.
Our  main result (Theorem~\ref{Theo1}) shows in particular 
that this is the only obstruction for the existence of a spectral gap 
when $H$ is a countable group of affine transformations of a compact 
nilmanifold $X$.

Let $N$ be a connected and simply connected nilpotent Lie group.
Let $\La$ be a lattice in $N;$ the associated nilmanifold
 $\nil$ is known to be compact.
The group $N$ acts by right translations on
$\nil:$  every $n\in N$ defines a transformation
$\rho(n)$ on $\nil$ given by $\La x\mapsto \La xn.$
Denote by $\Aut(N)$  the group  of  continuous automorphisms of $N$
and by  $\Aut (\nil)$ the subgroup
of continuous automorphisms  $\varphi$ of $N$ such that $\varphi (\La) =\La.$
The group $\Aut(N)$ is   a linear algebraic group
defined  over $\mathbf Q$ and  $\Aut (\nil)$
is a discrete subgroup of $\Aut(N).$
An affine transformation of $\nil$  is a
mapping $\nil\to \nil$ of the form $\vfi \circ \rho(n)$
for some $\vfi\in \Aut(\nil)$ and $n\in N.$
The group $\Affnil$ of affine transformations of $\nil$
is the semi-direct product $\Aut (\nil)\ltimes N.$

Every $g\in \Affnil$  
preserves the translation invariant probability
measure $\nu_{\nil}$ induced by a Haar measure
on $N.$  The action of   $\Affnil$  on $\nil$ is a  natural
generalization of  the action  of $SL_n(\mathbf Z)\ltimes \TT ^n$
on the torus $\TT ^n=\RRR^n/\ZZ^n.$
In fact,  let $T=\tor$ be the maximal
torus factor of $\nil.$ Then the nilsystem $(\nil, H)$  can be viewed
as  the result, starting with  $T,$  of a finite sequence 
of extensions by tori, with induced actions of $H$ on  every stage.

Actions of  of higher rank lattices by affine transformations
on nilmanifolds arise  in Zimmer's programme
 as one  of the standard actions for such groups
(see the survey \cite{Fisher}).
 The action of a single  affine transformation (or a flow of 
such transformations) on  a nilmanifold have been
studied by W.~Parry from the  ergodic, spectral or topological 
point of view (see \cite{Parry1},\cite{Parry2},\cite{Parry3}; see
also \cite{AGH} for the case of  translations).

Let  $V$ be a finite
dimensional real vector space
and  $\Delta$ a lattice in $V.$
As is well-known, $T=V/\Delta$ is a torus and $\Delta$ defines a 
rational structure on $V.$ Let $W$ be a
rational linear subspace of $V$. 
Then  $S= W/ (W\cap \Delta)$ is a subtorus  of $T$ 
and we have a torus factor  $\overline T= T/S.$ 
 Let $H$ be a subgroup
of $\Aff (T)$ and assume that $W$ is invariant
under $\paut(H)$, where   $\paut:\Affnil\to \Autnil$
is the canonical projection.
Then  $H$  leaves  $S$  invariant and 
the induced action of $H$ on $\overline T$ is a  factor of 
the action of $H$ on $T.$ 
We will say that $\overline T$ is an 
$H$-invariant factor torus of $T.$   Here is our main result.

\begin{theorem}
\label{Theo1}
Let $\nil$ be a compact  nilmanifold  
with associated maximal torus factor $T.$ 
Let  $H$ be a countable subgroup $\Affnil$.
 The following properties are equivalent:
\begin{itemize}
 \item [(i)] The action of $H$ on $\nil$  has a spectral gap.
 \item[(ii)] The action of  $H$  on  $T$  has a spectral gap.
\item [(iii)] There exists no  non-trivial  $H$-invariant factor torus   $\overline T$ of $T$    such that
the projection of  $\paut (H)$ on $\Aut(\overline T)$ is a virtually abelian group (that is,  it  contains an abelian subgroup of finite index).
\end{itemize}
\end{theorem}

To give an  an example, let  $T= \RRR^d/\ZZ^d$ be the $d$-dimensional torus.
Observe that   $\Aut (T)$ can be identified with $GL_d(\ZZ)$.
Let $H$ be a subgroup of $\Aff(T)=GL_d(\ZZ)\ltimes T$.
Assume that $\paut(H)$ is not virtually abelian and that 
$\paut(H)$ acts   $\QQ$-irreducibly  on $\RRR^d$
(that is, there is no non-trivial  $\paut(H)$-invariant rational subspace
of $\RRR^d$). Then the action of $H$ on $T$ has a spectral gap.
For more details, see Corollary~\ref{Cor-QIrred-Torus}
and Example~\ref{Exa1} below.

The result above is new even in the case where $\nil$ is a torus;
see however \cite[Theorem~6.5.ii]{FuSh} for a sufficient condition
for the existence of a spectral gap for groups of torus automorphisms.
 Our results  shows, in particular, that
 the spectral  gap property for a countable subgroup
 $H$ of $\Affnil$ is equivalent to the spectral gap
 property for its automorphism part $\paut (H).$  
   
The proof of Theorem~\ref{Theo1} breaks into  two parts. We first  establish
 the result in the case  where $\nil$ is a torus (see Theorem~\ref{Theo3} below ). 
 Our proof  is based here on
 the  existence of appropriate invariant means
 on finite dimensional vector spaces.  A crucial tool 
 will be  (a version of) Furstenberg's result
 on stabilizers of probability measures on projective spaces over local fields.
  In the case of a general nilmanifold
 $\nil$ with associated  maximal torus factor $T,$ 
we  show that (ii) implies (i) by studying
  the asymptotic behaviour  of matrix coefficients  of 
the Koopman representation $U$ of  $H$ 
restricted to the orthogonal   complement of $L^2(T)$ in
 $L^2(\nil)$; for this, we will use 
 decay properties of the metaplectic representation 
 of symplectic groups due to R.~Howe and C. C.Moore  \cite{HoMo}.
The equivalence of (i) and (ii) was proved
in \cite{BeHe}  in the special case
of a group of automorphisms of Heisenberg nilmanifolds.

Actions of countable amenable groups on a 
 non-atomic probability space fail to 
 have a  property which is weaker 
 than the spectral gap property. Recall that the action of a  countable group $H$
 by measure preserving transformations on a probability space $(X, \nu)$
 is said to be \emph{strongly ergodic} in Schmidt's sense
 (see \cite{Schmidt1}, \cite{Schmidt2}) if every
 sequence $(A_n)_n$ of measurable subsets of  $X$ which is asymptotically invariant
 (that is,  which is such that  $\lim_n\nu(g A_n \bigtriangleup A_n)=0$ for all $g\in H$)
is trivial (that is, $\lim_n\nu( A_n)(1-\nu(A_n))=0$).
It is easy to see that if the action
of $H$ on $X$ has a spectral gap, then the action is
strongly ergodic (see, for instance, \cite[Proposition 6.3.2]{BHV}).
The converse does not hold in general (see Example (2.7) in  \cite{Schmidt2}).
As shown in \cite{Schmidt2}, no action of a countable {amenable}
 group by measure preserving transformations  on a non-atomic, 
  probability space can be  strongly ergodic.

An interesting feature
of  strong ergodicity (as opposed to the spectral gap property) is that this notion
only depends on the equivalence relation
on $X$ defined by the partition of $X$  into $H$-orbits.
Our result shows that the existence of a spectral gap
for subgroups of $\Affnil$ is equivalent to strong ergodicity.

\begin{corollary}
\label{Cor1}
The action of a countable subgroup
of $\Affnil$ on a  compact  nilmanifold $\nil$
 has a spectral gap if and only if it  is strongly ergodic.
\end{corollary}

We suspect that the previous corollary is true
for every countable group  of affine transformations  of
the quotient of a Lie group  by  a lattice. In fact,
the following stronger statement could be true.
Let $G$  be  a connected  Lie group
and $\Ga$ a lattice of $G$.  Let 
$H$ be a countable subgroup of  ${\rm Aff}(\Ga\bs G)$.
Assume that  the action of $H$ on $\Ga\bs G$
 does not have  a spectral gap. Is it true that 
 there exists a non-trivial $H$-invariant factor  $\overline{\Ga}\backslash\overline{G}$
 of $\Ga\bs G$ such that the 
 closure of the projection of $H$ on ${\rm Aff}(\overline{\Ga}\backslash\overline{G})$
 is an amenable group?
 
 As our result shows, this  is indeed the case
 if $G$  is a nilpotent Lie group; it  is also the case
 if $G$  is a simple non-compact Lie group 
 with finite centre  (see Theorem 6.10 in \cite{FuSh}).
It is worth mentioning   that the corresponding  statement in the framework
 of countable standard equivalence relations 
 has been proved in  \cite{JoSch}.

 Let again $H$ be a  countable group acting 
 by measure preserving transformations on a probability space $(X,\nu).$ 
 The following useful  strengthening of the spectral gap property has been considered
 by several  authors (\cite{Bekka}, \cite{BeGui}, \cite{FuSh}, \cite{Popa}). 
 Following  \cite{Popa}, let us say that the  action of $H$ has a \emph{stable spectral gap}
 if the diagonal  action of $H$ on $(X\times X, \nu\otimes\nu)$ has
 a spectral gap (see Lemma 3.2 in \cite{Popa} for
 the rationale of this terminology).  The following result is an immediate
 consequence of Theorem~\ref{Theo1} above and 
 of the corresponding result for    groups of torus automorphisms 
obtained in \cite[Theorem 6.4]{FuSh}.

\begin{corollary}
\label{Cor2}
If the action of a countable subgroup
of $\Affnil$ on a  compact  nilmanifold $\nil$
 has a spectral gap, then  it  is  has 
 stable spectral gap.
 \end{corollary}

Next, we turn to the question of the ergodicity or mixing 
of  the action of a (not necessarily countable) subgroup $H$
of $\Affnil$  on $\nil$.
As a consequence of our methods, 
we will see that this  reduces to the same question for the
action of   $H$ on the associated torus. 

Recall that an action of a  group $H$  on a probability space $(X,\nu)$ is  weakly mixing if the Koopman representation $U$ of $H$ on $L^2(X,\nu)$
has no finite dimensional subrepresentation, and that the action of 
of a countable group $H$
is strongly mixing if the matrix coefficients $g\mapsto \langle U (g)\xi, \eta\rangle$ vanish at infinity for all $\xi,\eta\in L^2_0(X,\nu).$
\begin{theorem}
\label{Theo2} 
Let $H$ be a   group of affine transformations of the compact nilmanilfold
 $\nil.$   Let $T$ be the  maximal $T$ torus factor  associated to  $\nil.$
\begin{itemize}
 \item [(i)] If the action of  $H$ on $T$ is ergodic (or weakly mixing), then  its
action on $\nil$ is ergodic (or weakly mixing).
 \item [(ii)] Assume that $H$ is as subgroup of $\Autnil.$ If the action of  $H$ on $T$ is strongly mixing, then  its action  on $\nil$ is strongly mixing.
\end{itemize}
\end{theorem}

Part (i) of the  previous theorem has been independently established  in  \cite{CoGu2}) with a different method 
of proof.  In the case of a single  affine transformation
 (that is, in the case of $H=\ZZ),$   the result  is due  
 to W.Parry (see \cite{Parry1}, \cite{Parry2}).
  Also,  \cite{CoGu2} gives  an example
of a group of automorphisms $H$
acting ergodically  on a nilmanilfold $\nil$
for which no single automorphism from $H$ acts ergodically
on $\nil,$ showing that the previous theorem
does not follow from Parry's result.

Sections 1-7 are devoted to the proof 
our  main result Theorem~\ref{Theo1}
in the case where $\nil$ is a torus.
The  proof of the extension to general nilmanifold
is given in  Sections 8-14.
Theorem~\ref{Theo2} is treated in Section~15.

\medskip
\noindent
\textbf{Acknowlegments} We are grateful to  J-P.~Conze, A.~Furman, and A.~Gamburd 
for useful discussions.

\section{Spectral gap  property for groups of affine transformations of a torus: statement of the main result}
\label{S1}

Let $V$ be a  finite dimensional  real vector space of dimension $d\geq 1$ and let 
$\Delta$  be a lattice in $V.$ Let $T$ be the torus
$T=V/\Delta.$  The group of affine transformations of $T$
is  the semi-direct product $\Aff(T)=\Aut (T)\ltimes T$.

The aim of this section is to state the following result,
which will be proved in the next two sections. Recall that $\paut$
denotes the canonical homomorphism  $ \AffT\to \AutT$.

\begin{theorem}
\label{Theo3}
Let  $H$ be a countable subgroup of $\AffT.$
The following properties are equivalent.
 The following properties are equivalent:
\begin{itemize}
 \item [(i)] The action of $H$ on $T$ does not have a spectral gap.
\item [(ii)] There exists a  non-trivial $H$-invariant factor torus $\overline T$
such that the projection  of  $\paut(H) $ on $\Aut (\overline T)$  is 
amenable.
\item [(iii)] There exists a  non-trivial $H$-invariant factor torus $\overline T_0$  such that
the projection  of  $\paut(H)$ on $\Aut(\overline T_0)$  is 
virtually abelian.
\end{itemize}

\end{theorem}
The following corollary is an immediate consequence
of the implication $(i)\Rightarrow (iii)$ in the previous theorem.
\begin{corollary}
\label{Cor-QIrred-Torus}
Let $T=V/\Delta$ be a torus. Let $H$ be a countable subgroup of $\AffT$ such that 
$\paut(H)\subset \Aut (T)$ is not virtually abelian. 
Assume that the action of $H$ on $V$ is $\QQ$-irreducible
for the rational structure on $V$ defined by $\Delta.$
Then the action of  $H$ on $T$ has a spectral gap.
\end{corollary}
This last result  
was proved in   \cite[Theorem~6.5.ii]{FuSh} for a 
subgroup $H$ of  $\Aut(T)$ under the stronger assumption
that the action of $H$ on $V$ is $\RRR$-irreducible.
We give an example  of a subgroup $H$ 
of automorphisms of a $6$-dimensional torus $T=V/\Delta$
 which acts $\QQ$-irreducibly but not $\RRR$-irreducibly
on $V$ and which has a spectral gap on $T.$
\begin{example}
\label{Exa1}
Let $q$ be the quadratic form on $\RRR^3$ given by 
$$q(x)=x_1^2+ x_2^2- \sqrt{2}x_3^2,$$
 and let $SO(q,\RRR)\subset GL_3(\RRR)$
be the orthogonal group of $q.$ Set 
$$H=SL_3(\ZZ[\sqrt{2}]\cap SO(q,\RRR).$$
Let $\sigma$ be the  non-trivial automorphism
of the field $\QQ[\sqrt{2}].$  For every $g\in  SO(q,\RRR),$
the matrix $g^\sigma$, obtained by conjugating each entry of $g,$ preserves  the conjugate
form  $q^\sigma$  of $q$ under $\sigma.$ 
The mapping 
$$
\QQ[\sqrt{2}]  \to \RRR\times \RRR, \qquad x\mapsto (x,\sigma(x))
$$
induces an isomorphism between  $\ZZ[\sqrt{2}]^3$ and a lattice $\Delta$ in 
$\RRR^3\times\RRR^3.$ 
It induces also   an isomophism
$\ga\mapsto (\ga,\ga^\sigma)$
between $H$ and a lattice  $\Ga$ in  $SO (q,\RRR)\times SO (q^\sigma,\RRR).$
 Moreover, $H$ leaves $\ZZ[\sqrt{2}]^3$ invariant  and 
$\Ga$ leaves $\Delta$ invariant. 
We obtain in this way an action of $H$ on the torus $T= \RRR^6/\Delta.$

Since $SO(q^\sigma,\RRR)\cong SO(3)$ is compact,  $H$ is a lattice
in $SO (q,\RRR).$ This implies (Borel density theorem) that 
the Zariski closure of $H$ in $SL_3(\RRR)$ is 
the simple Lie group $SO(q,\RRR),$
so that the action of $H$ on $\RRR^3$ is 
$\RRR$-irreducible and hence $\QQ$-irreducible
for the usual rational structure on $\RRR^3.$
It follows that the action of $H$ on $\RRR^6$ is $\QQ$-irreducible for the
rational structure defined by the lattice $\Delta$ of $\RRR^6.$
Observe that the action of $H$ on $\RRR^6$ is not $\RRR$-irreducible 
since $\Ga$ leaves invariant  each copy  of $\RRR^3$  
in $\RRR^6= \RRR^3\oplus \RRR^3.$
Moreover, $H$ is not virtually abelian as it is a lattice in
$SO (q,\RRR)\cong SO(2,1).$  As a consequence of the previous corollary,
the action of $H$ on $T$ has a spectral gap.
\end{example}

Concerning the proof of Theorem~\ref{Theo3},
we will first treat the case of groups of toral automorphisms.

Choosing  a basis for the $\ZZ$-module  $\Delta,$
we identify $V$ with $\RRR^d$ and $\Delta$ with $\ZZ^d.$
By means of the standard scalar product on $\RRR^d,$
we identify 
the  dual group $\widehat V$  of
   $V$ (that is, the group of unitary characters
of $V$) with  $V$. 
The dual action of an element $g\in GL(V)$ on $\widehat V$ corresponds to the 
action of $(g^{-1})^t$ on $V.$
Since  $T=V/\Delta,$ the  dual group $\widehat T$ 
 can be identified with  $\Delta.$
Let $W$ be a
rational linear subspace of $V$.
The dual group of the quotient $V/W$  corresponds
to the orthogonal complement $W^\perp$  of $W,$
which is also a rational linear subspace of $V$.
The dual group of the torus factor  $\overline T= (V/W)/((W+\Delta)/\Delta)$ 
corresponds to $W^\perp \cap \Delta.$

The discussion above shows that   Theorem~\ref{Theo3},
in the case of a group of toral automorphisms is equivalent to the following theorem.

\begin{theorem}
\label{Theo4}
Let  $H$ be a subgroup of $GL_d(\ZZ).$
The following properties are equivalent.
\begin{itemize}
 \item [(i)] The action of $H$ on $T=\RRR^d/\ZZ^d$ does not have a spectral gap.
\item [(ii)] There exists a  non-trivial  rational subspace $W$
of $\RRR^d$ which is  invariant  under  the subgroup  $H^t$
of $GL_d(\ZZ)$  and such that  the image of $H^t$ in $GL(W)$ is 
an amenable group.
\item [(iii)] There exists a  non-trivial  rational subspace $W$
of $\RRR^d $ which is invariant under $H^t$    and 
such that the  the image of $H^t$   in $GL(W)$ is 
a virtually abelian  group.
\end{itemize}

\end{theorem}

Observe that  the implication $(iii) \Longrightarrow (ii)$ is obvious and that
the implication $(ii) \Longrightarrow (i)$  follows from the result
in \cite{JuRo} quoted in the introduction.
Therefore, it remains to show that (i) implies (ii) and that (ii) implies (iii).

\section{A canonical amenable group associated to  a linear group }
\label{S10}
Let $V$ be a finite-dimensional  real vector space.
(Although we will consider only real vector spaces,
the results in this section are valid for vector spaces
over any local field.)
 Let $g\in GL(V)$ and $ W$ a $g$-invariant  linear subspace
of $V.$ We denote by $g_W\in GL(W)$ the automorphism of $W$
given by the restriction  of $g$ to $W.$
If $W'$ is  another $g$-invariant subspace contained in $W,$
we will denote by $g_{W/W'}\in GL(W/W')$ the automorphism
of $W/W'$ induced by $g.$
Also, if $H$ is a subgroup of $GL(V)$ and $W'\subset W$ are $H$-invariant subspaces of $V,$
 we will denote by $H_W$ and $H_{W/W'}$ the corresponding subgroups
of $GL(W)$ and $GL(W/W'),$ respectively.

 For a subgroup $H$ of $GL(V),$ we denote by $\overline H$
its closure for the usual locally compact  topology on  $GL(V).$ 
The aim of this section is to prove the following  result. 
\begin{proposition}
 \label{Pro-AmenableQuotient}
Let $H$ be a  subgroup of $GL(V).$ There exists a largest $H$-invariant 
linear subspace $V(H)$
of $V$ such that the group $\overline{H_{V(H)}} $ is amenable.
More precisely,  let $V(H)$ be the subspace of $V$ generated by the union
of the $H$-invariant subspaces $W\subset V$ for which  $\overline{H_W}$ is amenable.
Then  $\overline{H_{V(H)}}$ is amenable.
\end{proposition}

A more  explicit  description of  $V(H)$  will be given later  (Proposition~\ref{Pro-AmenableQuotient-Bis}).
 For the proof of  the proposition above, we will need  the
following elementary lemma.

\begin{lemma}
 \label{Lem1-AmenableQuotient}
Let $H$ be a closed subgroup of $GL(V)$ and $W$ an $H$-invariant subspace of $V.$
 Then  $H$ is amenable if and only if
$\overline{H_{W}}$ and $\overline{H_{V/W}}$ are amenable. 
\end{lemma}
\begin{proof}
Since $\overline{H_{W}}$ and $\overline{H_{V/W}}$
  are  closures of quotients of  $H,$
both are amenable if $H$ is amenable.

Assume that $\overline{H_{W}}$ and $\overline{H_{V/W}}$ are amenable. 
Let $L$ be the closed subgroup consisting of the elements $g\in GL(V)$
leaving $W$ invariant and
for which $g_{W}$ belongs to  $\overline{H_{W}}$ and $g_{V/W}$ belongs to $\overline{H_{V/W}}.$
The mapping  
$$
\vfi: L\to \overline{H_{W}}\times \overline{H_{V/W}},\qquad g\mapsto (g_{W}, g_{V/W})
$$
is a continuous homomorphism. It is clear that $\vfi$ is surjective.
Moreover, $U=\Ker(\vfi)$ is a unipotent closed subgroup of $L.$
Since $\overline{H_{W}}\times \overline{H_{V/W}}$ and $U$ are amenable, $L$ is amenable.
The closed subgroup $H$ of $L$ is therefore amenable.$\bsq$
 
\end{proof}

\n
\textbf{Proof of Proposition~\ref{Pro-AmenableQuotient}} \ 
We can  write  $V(H)= \sum_{i=1}^r W_i$ as a sum
of  finitely many $H$-invariant subspaces
$W_1,\dots, W_r$ of $V$ such that  $\overline{H_{W_i}}$ is amenable for every $1\leq i\leq r.$

We show by induction on $s\in\{ 1, \dots, r\}$ that
 $\overline{H_{W^s }}$ is amenable, where
$W^s=\sum_{i=1}^{s} W_i.$ The case $s=1$ being obvious,
assume that $\overline{H_{W^s}}$ is amenable for some $s\in\{ 1, \dots, r-1\}.$
The group $$GL(W^{s+1}/ W^s)= GL((W^{s}+ W_{s+1})/ W^s)$$ is canonically isomorphic
to $GL(W_{s+1}/(W^s\cap W_{s+1})$
and $\overline{H_{W^{s+1}/ W^s}}$ corresponds to  $\overline{H_{W_{s+1}/(W^s\cap W_{s+1})}}$
under this isomorphism. 
Now, $\overline{H_{W_{s+1}/(W^s\cap W_{s+1})}}$ is amenable since $\overline{H_{W_{s+1}}}$
is amenable. Hence,  $\overline{H_{W^{s+1}/ W^s}}$ is amenable.
Moreover,  $\overline{H_{W^s }}$ is amenable by the induction hypothesis.
The previous lemma implies  that $\overline{H_{W^{s+1}}}$
is amenable. $\bsq$


\section{Invariant means  supported by  rational subspaces}
\label{S11}
Let $G$ be a locally compact group. 
There is a well-known relationship
between  weak containment
properties of the trivial  representation $1_G$ and existence on invariant means
on appropriate spaces (see below).
We will need to make this relationship more precise in the case
where $H$ is a subgroup of toral automorphisms.

By a  unitary representation $(\pi,\H)$
of $G$, we will always mean a strongly continuous homomorphism $\pi:G\to U(\H)$ 
from $G$ to the unitary group of a complex Hilbert space $\H.$

Recall that, for every finite measure $\mu$
of $G,$ the operator $\pi(\mu)\in \B(\H)$ is defined
by  the integral 
$$
\pi(\mu) \xi= \int_G \pi(g) \xi d\mu(g) \tout \xi\in \H.
$$
Assume that   $G$ is a discrete group and  $\pi$ and $\rho$ are unitary representations of
$G$;  then $\pi$ is weakly contained in $\rho$ if and only if
$\Vert \pi(\mu)\Vert \leq \Vert \rho(\mu)\Vert$ for every 
finite measure $\mu$ on $G$  (see Section~18 in  \cite{Dixmier}).
Recall also that, given a  probability measure
$\mu$ on $G$ which is  aperiodic,
the trivial representation $1_G$
 is weakly contained in a unitary representation $\pi$
  if  and only if $\Vert \pi(\mu)\Vert=1$  (see \cite[G.4.2]{BHV}).

Let $X$ be a topological space and $C^b(X)$ the Banach
space of all bounded continuous functions on $X$ equipped
with the supremum norm. Recall that a mean on $X$ is a 
linear functional $m$ on $C^b(X)$ such that
$m(1_X)=1$ and such that $m(\vfi)\geq 0$ for every 
$\vfi\in C^b(X)$ with $\vfi\geq 0.$ A mean is automatically
continuous. 
We will often write $m(A)$ instead of $m(1_A)$ for a subset $A$ of $X.$

 Observe that the means on a compact space $X$
are the probability measures on $X.$

Let $H$ be a group acting on $X$ by homeomorphisms. 
Then $H$ acts naturally on $C^b(X).$
A mean $m$ on $X$ is $H$-invariant if 
$m(h.\vfi)= m(\vfi)$ for all $\vfi\in C^b(X)$ and $h\in H.$

Let $Y$ be another topological space and $f:X\to Y$   a
continuous mapping. For every mean $m$ on $X,$ the push-forward
$f_*(m)$ of $m$ is the mean on $Y$ defined by 
$\vfi\mapsto m(\vfi\circ f )$ for $ \vfi\in C^b(Y)$.

We will consider invariant means on two kinds of topological spaces:

\n
$\bullet$ $X$ is a set with the  discrete topology
and endowed with an action of a group $H$.   It is well-known (see Th\'eor\`eme on p. 44 in \cite{Eymard}) that 
there exists an $H$-invariant mean on $X$ if and only if
the natural unitary representation $U$ of $H$ on $\ell^2(X)$
almost has invariant vectors (that is, if and only if $U$ weakly contains
the trivial representation $1_H$ of $H$).

\n
$\bullet$ $X=V\setminus\{0\},$ where $V$ is a finite dimensional 
real vector space. Let $H$ be a subgroup of $GL(V).$
If $m$ is an $H$-invariant  mean
on $V\setminus\{0\},$ then $\pi_*(m)$ is an $H$-invariant probability measure
on the projective space ${\mathbf P}(V),$ where 
$\pi: V\setminus\{0\} \to {\mathbf P}(V)$ is the canonical projection.

The following result is a version of Furstenberg's celebrated lemma
(see \cite{Furst} or  \cite[Corollary 3.2.2]{Zimmer})
on stabilizers of probability measures on projective spaces.
We will need later (in Section~\ref{S13}) the  more precise form we give  for this
lemma (see also the proof of  Theorem 6.5 (ii) in \cite{FuSh}).

For a subgroup $H$  of $GL(V),$  we denote by 
$\ZC(H)$ the  closure of $H$ in the Zariski topology 
 and by $\ZC(H)^0$ the connected component of $\ZC(H)$ in the Zariski topology.
As is well-known, $\ZC (H)^0$ has finite index in $\ZC(H).$

\begin{lemma}
\label{Lem-Furstenberg}
Let $H$ be a closed subgroup of $GL(V).$
Assume that  $H$ stabilizes a probability
measure $\nu$ on ${\mathbf P}(V)$ which is not supported on a proper
projective subspace. Then the commutator subgroup $[H^0, H^0]$ 
of $H^0$ is relatively compact, where $H^0$ is the normal subgroup of finite index
$H\cap \ZC(H)^0$ of $H.$ In particular, $H$ is amenable.
\end{lemma}
\begin{proof}
We can find finitely many  positive measures $(\nu_i)_{1\leq i\leq r}$ on 
 ${\mathbf P}(V)$ with  $\nu=\sum_{1\leq i\leq r}\nu_i$
  such that $\nu(V_i\cap V_j)= 0$ for $i\neq j$ and such that $\supp(\nu_i)\subset  \pi(V_i)$ for every $i\in\{1,\dots, r\},$
 where $V_i$ is a linear subspace of $V$ of minimal dimension with
 $\nu_i( \pi(V_i))>0.$ 
 The $H$-orbit of $V_i$ and hence the $H$-orbit  of $\nu_i$ is finite
 (see  Proof of  Corollary 3.2.2 in \cite{Zimmer}).
 Since stabilizers of probability measures on  ${\mathbf P}(V)$ 
 are algebraic  (see  Theorem  3.2.4 in \cite{Zimmer}), 
 it follows that $H^0$ stabilizes each $V_i$ and each $\nu_i.$  
 Now $\nu_i$, viewed as measure on  ${\mathbf P}(V_i)$,
 is zero on every proper projective subspace of   ${\mathbf P}(V_i)$.
 Hence (see  Corollary 3.2.2 in \cite{Zimmer}),  the image of the  restriction $H_i^0$ of $H^0$
 to   $V_i$ is a relatively compact subgroup of $PGL( V_i),$   for every $i\in\{1,\dots, r\}.$
 Since $[H_i^0,H_i^0]$ is contained in $SL(V_i),$ it follows that 
  $\overline{[H_i^0,H_i^0]}$ is  compact in   $GL(V_i).$
  This implies that $\overline{[H^0,H^0]}$  is compact.
  As  $H^0/ \overline{[H^0,H^0]}$ is abelian, it follows that 
 $H^0$ (and hence $H$) is amenable.$\bsq$
\end{proof}

\begin{remark} The conclusion of  the previous lemma
 does not hold in general if we replace  $H^0$ by
an arbitrary subgroup of finite index of $H.$
For example, let $V= \RRR e_1 \oplus \RRR e_2$ and let 
$H\subset GL_2(\RRR)$ be the stabilizer of the measure
$\nu= (\delta_{\pi(e_1)} + \delta_{\pi(e_2)})/2$ on ${\mathbf P}(V).$
Then $[H,H]=H$ is not bounded; however, $H^0$ is the 
subgroup of index two consisting 
of the  diagonal matrices in $H$ and $[H^0, H^0]$ is trivial.
\end{remark}
 
\begin{proposition}
 \label{Pro-InvMean}
Let $H$ be a subgroup of $GL(V)$ and $V(H)$ the largest $H$-invariant
susbpace of $V$ such that $\overline{H_{V(H)}}$ is amenable.
\begin{itemize}
 \item [(i)] Assume $H$ stabilizes a mean $m$ on $V\setminus \{0\}.$
Then $V(H)\neq \{0\}.$
\item[(ii)] Let $\Delta$ be a lattice in $V$ and $m$  a mean
on $\Delta\setminus \{0\}.$ Assume $H$ leaves $\Delta$ invariant and
stabilizes $m.$ Then $m (V(H)\cap \Delta) =1.$ In particular,
the $\RRR$-linear span of $V(H)\cap \Delta$  is a non-trivial rational 
subspace of $V$ (for the rational structure defined by $\Delta$).
\end{itemize}
\end{proposition}
\begin{proof}
(i) Let $\pi: V\setminus\{0\} \to {\mathbf P}(V)$ be
 the canonical projection and $\nu=\pi_*(m).$  Then $\nu$
is an $H$-invariant probability measure on ${\mathbf P}(V).$
Let $W$ the linear span of  $\pi^{-1} (\supp(\nu)).$
Then $W$ is non-trivial and  $\nu$ is not supported on a proper
projective subspace of $\pi(W).$ It follows from Lemma~\ref{Lem-Furstenberg}
applied to the closed subgroup  $\overline{H_W}$ of $GL(W)$ that 
$\overline{H_W}$ is amenable. Hence, $V(H)\neq \{0\}$, by the definition
of $V(H).$

(ii) 
Set  $\overline V= V/V(H).$ 
Since $V(H)$ is $H$-invariant, we have an induced action of
$H$ on $\overline V.$ Denote by $p: V\to \overline{V}$
the canonical projection.
We consider the mean $\overline{m}=(p|_\Delta)_*(m)$ on the set 
$\overline{\Delta}:=p(\Delta)$ equipped with the discrete topology.
Observe that  $\overline{m}$ is $H$-invariant, since $H$ stabilizes $m.$

Assume, by contradiction, that $m (V(H)\cap \Delta)<1.$
Then $\overline{m}(\{0\}) = m (V(H)\cap \Delta)<1.$
Setting $\alpha= m (V(H)\cap \Delta),$ we define an $H$-invariant mean
$\overline{m_1}$ on $\overline{\Delta} \setminus \{0\}$ by 
$$
\overline{m_1} (\vfi) = \frac{1}{1-\alpha} \overline{m} (\vfi)\tout \vfi\in \ell^\infty(\overline{\Delta} \setminus \{0\}).
$$
Let  $i_*(\overline{m_1})$ be the mean on $\overline{V}\setminus \{0\}$
induced by the canonical injection 
$i: \overline{\Delta}\setminus \{0\} \to \overline{V}\setminus \{0\}.$
Observe that  $i_*(\overline{m_1})$ is $H$-invariant. Hence, by 
(i), we have $\overline{V}(H)\neq \{0\}.$
This implies  that $V(H)$ is a proper subspace of
the  vector space $W:=p^{-1} (\overline{V}(H)).$
On the other hand, $\overline{H_W}$ is amenable, by Lemma~\ref{Lem1-AmenableQuotient}.
This contradicts the definition of $V(H).$ $\bsq$
\end{proof}

At this point, we can give the proof of the fact that
(i) implies (ii) in  Theorem~\ref{Theo3} (or, equivalently,
in Theorem~\ref{Theo4}) in the case of  group of automorphisms.

\n
\textbf{Proof of   $(i) \Longrightarrow (ii)$   in Theorem~\ref{Theo4}}

Let  $H$ be a countable subgroup of $GL_d(\ZZ).$
Assume that  the action of $H$ on $T=\RRR^d/\ZZ^d$ does not have a spectral gap.
Then the unitary representation of 
the transposed subgroup $H^t$ on $\ell^2 (\ZZ^d\setminus\{0\})$ 
weakly contains the trivial representation $1_{H^t}.$
Hence, there exists an $H^t$-invariant mean on $\ZZ^d\setminus\{0\}.$
By Proposition~\ref{Pro-InvMean}, the linear span $W$ of $V(H^t)\cap \ZZ^d$  is a non-trivial rational 
subspace of $\RRR^d.$  Morever, $H^t_W= \overline{H^t_W}$ is amenable. $\bsq$

\section{Proof   of $(ii) \Longrightarrow (iii)$   in Theorem~\ref{Theo4}} 
\label{S13}

For the proof  of $(ii) \Longrightarrow (iii)$   in Theorem~\ref{Theo4},
we will need a precise description of the subspace
$V(H)$ associated to a  subgroup
$H$ of $GL(V)$  and introduced in  Proposition~\ref{Pro-AmenableQuotient}.
For this, we will use the following  result  which appears as Lemma~1
and Lemma~2 in 
\cite{CoGu}.  Since the  arguments  in  \cite{CoGu}
are slightly incomplete, we give the proof of this lemma.

\begin{lemma}
\label{Lem-CoGu}
Let  $V$ be  finite-dimensional real vector space
and let $H$ be a subgroup of $GL(V)$  such that the
action of $H$ on $V$ is completely reducible.
\begin{itemize} 
\item[(i)]
Assume that the  eigenvalues of every 
element in $ H$ all have modulus 1. Then $H$ is relatively compact.
\item[(ii)] Assume that there exists an integer $N\geq 1$ such that the  eigenvalues of every 
element in $ H$  are all $N$-th roots of unity. Then $H$ is finite.
\end{itemize}
\end{lemma}
\begin{proof}
\noindent
By hypothesis, we can 
decompose $V$ into a direct sum  $V=\oplus_{1\leq i\leq r} V_i$
of irreducible $H$-invariant subspaces $V_i.$ 
Let $V^\CCC=V\otimes_\RRR \CCC$ be the
complexification of $V.$ The action of $H$ on 
each $V_i$ extends to 
a representation   of $H$ on $V_i^\CCC$
which either is  irreducible or decomposes
as a direct sum of two irreducible (mutually conjugate)
representations  of  $H.$
It suffices therefore to prove the following

\n
\textbf{Claim:} 
Let $H$ be a subgroup of $GL_d(\CCC)$ acting 
irreducibly on $\CCC^d.$ Then the conclusion (i) and (ii) hold.

 For every $h\in H,$ we consider the linear functional $\vfi_h$ on the 
algebra $M_d(\CCC)$ of  complex $(d\times d)$-matrices  defined by
$\vfi_h(x)= \Tr(hx).$
Since $H$ acts irreducibly, it follows from Burnside theorem
that the algebra generated by $H$ coincides with $M_d(\CCC).$
Hence, there exists a basis $\{h_1, \dots,  h_{d^2}\}$ of 
the vector space $M_d(\CCC)$ contained in $H.$
Then $\{\vfi_{h_1}, \dots, \vfi_{ h_{d^2}\}}$ is a basis of the dual
space of  $M_d(\CCC).$ 

Assume that the  eigenvalues of every 
element in $H$ all have modulus 1.  
Then the $\vfi_{h_i}$'s  are bounded on $H$ by  $d$.
 It follows that the matrix coefficients  of 
 the elements in $H$ are bounded. Hence,
 $H$ is  relatively compact  subset of   $M_d(\CCC).$

Assume that, for a fixed  $N\geq 1,$  the  eigenvalues of every 
element in $H$ are $N$-th roots of unity..  
Then the $\vfi_{h_i}$'s   take only a finite set
of values on  $H.$ 
 It follows that $H$ is finite subset of   $M_d(\CCC).\bsq$

\end{proof}

\begin{proposition}
\label{Pro-AmenableQuotient-Bis}
Let $V$ be a finite-dimensional real vector space 
and $H$  a   subgroup $H$ of $GL(V).$
Set $H^0= H\cap \ZC(H)^0$. 
Let $V^1$ be the largest  $H$-invariant 
linear subspace of $V$ such that, for every 
$h\in [H^0, H^0],$ the eigenvalues 
of the restriction of $h$  to $V^1$ all have 
modulus 1. Then $V(H)=V^1.$
\end{proposition}

\begin{proof}
Let us first show that $V(H)\subset V^1.$ Since
$\overline{H_{V(H)}}$ is amenable, there exists an 
$H$-invariant probability measure $\nu$ on ${\PP}(V(H)) \subset  \PP(V).$
Let $W$ be the smallest  $H$-invariant subspace such that
$\nu$ is supported on $\PP(W).$ 
It follows from  Lemma~\ref{Lem-Furstenberg} that  $ [H^0, H^0]$
acts isometrically on $W$, with respect to an appropriate norm on 
$W.$  
We can apply the same argument to the group $\overline{H_{V(H)/W}}$ acting on 
  the quotient space $V(H)/W.$  Hence, by induction, we obtain a flag
 $$\{0\}=W_0\subset W=W_1\subset  W_2\subset \cdots\subset W_r= V(H)$$
 of $H$-invariant subspaces such that  $ [H^0, H^0]$
acts isometrically on  each quotient $W_{i+1}/W_i$.
It follows from this that 
the eigenvalues of the restriction to $V(H)$
of  any element $h\in [H^0, H^0]$ have   all
modulus 1.  Hence,  $V(H)\subset V^1.$ 

To show that  $V^1\subset V(H),$  we have to prove that
$\overline{H_{V^1}}$ is  amenable. 
Recall that that  $H/H^0$ is finite and observe that $\overline{H^0_{V^1}}/\overline{[H^0_{V^1}, H^0_{V^1}]} $ 
is abelian. Hence, it suffices to show that  $\overline{[H^0_{V^1}, H^0_{V^1}]} $
is amenable.

Let 
$$\{0\}=W_0\subset W_1 \subset \cdots\subset W_r= V^1$$
be a Jordan-H\"older sequence for the $[H^0_{V^1}, H^0_{V^1}]$-module $V^1,$
that is, every $W_i$ is an $[H^0_{V^1}, H^0_{V^1}]$-invariant subspace of $V^1$
and   $[H^0_{V^1}, H^0_{V^1}]$ acts irreducibly on every quotient $W_{i+1}/W_i.$
By Lemma~\ref{Lem-CoGu}.i,   the image of $ [H^0, H^0]$
in   $GL(W_{i+1}/W_i)$ is relatively compact for every $i\in \{0,\dots, r-1\}.$

Let $N$ be the unipotent subgroup of $GL(V^1)$
consisting of the elements in  $GL(V^1)$ which act trivially
on every quotient   $W_{i+1}/W_i$.

We can choose a scalar product 
on $V^1$ such that, denoting by $W_i^\perp$ the orthogonal complement of $W_i$ in $W_{i+1},$
  every $h\in [H^0, H^0]$  can be written in the form 
$h=kh_0,$ where $h_0\in N$  and where
$k$  leaves 
$W_i^\perp$ invariant and  acts isometrically 
on   $W_i^\perp$ for every $i\in {\{0,\dots, r-1\}},$ 
This shows that  $\overline{[H^0_{V^1}, H^0_{V^1}]} $ 
can be embedded as a  closed subgroup of
$K\ltimes N \subset GL(V^1),$ where 
$K$ is the product of the 
 the orthogonal groups 
of  the $W_i^\perp$'s. 
 Since  $K\ltimes N$
is amenable, the same is true for    $\overline{[H^0_{V^1}, H^0_{V^1}]}. \  \bsq$
\end{proof}

We will need need the following corollary  of (the proof of)  the previous
proposition .

\begin{corollary}
\label{Pro-AmenableQuotient-Discret}
Let  $\Gamma$  be  a   subgroup  of $GL_d(\ZZ).$
Assume that  the eigenvalues  of every 
$\ga\in\Ga$  all have 
modulus 1. Then $\Ga$ contains  a unique maximal
unipotent subgroup  $\Ga^0$ of finite index. In particular, 
$\Ga^0$ is a characteristic subgroup of $\Ga.$
\end{corollary}

\begin{proof}
As in the proof of the previous proposition, 
we consider  a Jordan-H\"older sequence for the $\Ga$-module $\RRR^d$
$$\{0\}=W_0\subset W_1 \subset \cdots\subset W_r= \RRR^d$$
and let $N$  be the subgroup of all  $g\in GL(V)$ 
which act trivially on every $W_{i+1}/W_i.$ We choose a scalar product on $\RRR^d$ 
such that  $\Ga$  embeds as a subgroup of the semi-direct product 
$K\ltimes N$ for $K= \prod_{i=1}^d O(W_i^\perp),$
where $W_i^\perp$ is the orthogonal complement of $W_i$ in $W_{i+1}.$

Let $\ga\in \Ga.$  For every $l\geq 1,$
the $l$-th powers of the eigenvalues of $\ga$ are roots
of the same monic polynomial with integer coefficients
and of degree $d.$ 
Since the eigenvalues of $\ga$ are all of modulus 1,
the coefficients of this polynomial are bounded
by a number only depending on $d.$
By a standard argument (see e.g. 
the proof of Lemma~{11.6} in
\cite{Stewart}), it follows  that all the eigenvalues 
of $\ga$ are roots of unity of a fixed order $N$   
which only depends on $d.$

Let  $\overline\Ga$ be the projection of $\Ga$ in  $K.$ 
The action of $\overline\Ga$  is completely
reducible, since the $W_i^\perp$'s  are  irreducible, and 
it follows from Lemma~\ref{Lem-CoGu}.ii that $\overline \Ga$
is finite. Hence, $\Ga\cap N$ is a
unipotent  normal subgroup of finite index in $\Ga.$

We have therefore proved that $\Ga$  contains
a unipotent  normal subgroup of finite index.
We claim that $\Ga^0:= \Ga \cap \ZC(\Ga)^0$ is the unique
maximal unipotent  normal subgroup of finite index  in $\Ga.$

Indeed, let $\Ga_1$ be a   unipotent  normal subgroup of finite index  in $\Ga.$
Set  $U:=\ZC(\Ga_1).$  Observe that 
the connected component of $U$ coincides
with $\ZC(\Ga)^0,$ since $\Ga_1$ has finite index
in $\Ga.$   On the other hand, as is well-known,  $U$ is  connected
since it is a unipotent  algebraic group.
(Indeed, the Zariski closure of the subgroup generated by a unipotent element $u\in GL(\RRR^d)$
 contains the one-parameter subgroup through $u$; see e.g.  15.1. Lemma C in \cite{Humphreys}.)
 It follows that $\ZC(\Ga)^0=U$ is unipotent.
  Moreover, since $\Ga_1\subset U,$ we have $\Ga_1\subset \Ga^0$ and the
  claim is proved.  $\bsq$
  \end{proof}

We can now  complete  the proof of Theorem~\ref{Theo4}.

\n
\textbf{Proof of   $(ii) \Longrightarrow (iii)$   in Theorem~\ref{Theo4} }

Let $T=V/\Delta$ be a torus and 
$H$ a subgroup of $\Aut (T) \subset GL(V).$
Assume that there exists a  non-trivial  rational subspace $W$
of $V$ which is $H$-invariant and such that  
such that the restriction $H_W$ of $H$ to $W$ is 
an amenable group. In particular, we have $W\subset V(H).$

Set $H^0= H\cap \ZC(H)^0$.
By Proposition~\ref{Pro-AmenableQuotient-Bis},
 for every  $h\in [H^0, H^0],$ all the eigenvalues 
of the restriction of $h$  to $W$  have 
modulus 1. Since $W$ is rational, 
by the choice of a convenient  basis of $W,$ we can assume
that $\Ga: =[H^0, H^0] _W$ is a subgroup of $GL_d(\ZZ),$ 
where $d=\dim W.$   
It follows from Corollary~\ref{Pro-AmenableQuotient-Discret}
that  $\Ga$ contains  a
unipotent subgroup  $\Ga^0$ of finite index 
which is moreover characteristic.   Let $W_1$
be the space of the $\Ga^0$-fixed vectors in $W.$ 
Then $W_1$ is   a rational and non-trivial 
linear subspace of $W.$  Moreover,  $W_1$ is 
$H$-invariant, since $\Ga^0$ is characteristic.

We claim that $H_{W_1}$  is virtually abelian.
For this, it suffices to show that  $G:=H^0_{W_1}\subset GL(W_1)$ is virtually abelian. Observe first that    $[G,G]=\Ga_{W_1}$ 
is finite, since it is a quotient of the finite group $\Ga/ \Ga^0.$ 
Since $ [\ZC(G), \ZC(G)] \subset \ZC([G,G]),$ 
it follows that   $[\ZC(G), \ZC(G)]$ is finite.
On the other hand,  
the group $[\ZC(G)^0, \ZC(G)^0]$ is connected (see e.g. 
Proposition~17.2 in \cite{Humphreys}). Hence,  $\ZC(G)^0$ is
abelian. The subgroup $G\cap \ZC(G)^0$  has finite
index in $G$ and is abelian.
$\bsq$

\section{Herz's majoration principle for induced representations}      
\label{S-Herz}
Unitary representations of  a separable locally compact group
 $G$ induced by a closed subgroup $H$  will appear
several times in the sequel. We review their definition
when the homogeneous  space  $H\bs G$
has  $G$-invariant  measure.
This will always be the case in the situations we will encounter.
(Induced representation are still defined in the general case,
after appropriate change; see \cite{Mackey-Livre} or \cite{BHV}.)

Let $\nu$ be   non-zero $G$-invariant  measure  on $H\bs G.$
Let $(\sigma,\K)$ be a unitary representation
of $H.$ We will use the following model
for the induced representation $\ind_H^G \sigma$.
Choose a measurable section  $s: H\bs G\to G$ 
for the canonical projection $G\to H\bs G.$ 
Let $c:  (H\bs G)\times G\to H$ be the corresponding
cocycle defined by 
$$s(x)g= c(x,g) s(xg) \tout x\in  H\bs G,\ g\in G.
$$
The Hilbert space of $\ind_H^G \sigma$ is 
the space $L^2( H\bs G, \K)$  of 
all square-integrable measurable mappings $\xi: H\bs G\to \K$ and the 
 action of $G$ on $L^2( H\bs G, \K)$ is given by
$$
(\ind_H^G \sigma)(g) \xi(x)= \sigma (c(x,g))\xi(xg),\qquad g\in G,\ \xi\in L^2( H\bs G, \K),\ x\in G/H.
$$ 

In the sequel, we will use several times 
a well-known strengthening of Herz's  majoration principle from \cite{Herz}
concerning norms of convolution operators under an induced representation.
For an even more general version, see  \cite[2.3.1]{Claire}. 
For the convenience of the reader, we give  the short proof.
\begin{proposition} \textbf{(Herz's majoration principle)}
\label{Prop3}
\label{Pro-Herz}
Let $H$ be a closed subgroup of $G$ 
such that $H\bs G$ has a $G$-invariant Borel measure $\nu$  
and let  $(\sigma,\K)$  be a unitary representation
of $H.$  For every probability measure $\mu$ on 
the Borel subsets of $G,$
we have
$$\Vert (\ind_H^G\sigma)(\mu)\Vert\leq \Vert\rho_{G/H}(\mu)\Vert,
$$
where $\la_{G/H}$ is the natural representation of $G$ on $L^2(G/H).$
\end{proposition}
\begin{proof}
Let  $c: H\bs G\to H$ be
the cocycle defined by a Borel section of $H\bs G\to G.$
For $\xi\in L^2(H\bs G, \K, \nu),$ define $\vfi$ in the Hilbert space  $L^2(H\bs G, \nu),$
 of $\ind_H^G \sigma $
by $\vfi(x)=\Vert\xi(x)\Vert$  and observe that  
$\Vert \vfi\Vert= \Vert \xi\Vert.$
Using Jensen's inequality, we have
\begin{align*}
\Vert(\ind_H^G \sigma)(\mu) \xi\Vert^2 
&= \int_{H\bs G}\Vert(\ind_H^G(\mu)\xi(x))\Vert^2 d\nu(x) \\
&= \int_{H\bs G}\Vert\int_G \sigma (c(x,g))\xi(xg) d\mu(g)\Vert^2 d\nu(x) \\
&\leq \int_{H\bs G}\int_G \Vert\sigma (c(x,g))\xi(xg)\Vert^2 d\mu(g) d\nu(x) \\
&= \int_{H\bs G}\int_G \Vert\xi(xg)\Vert^2 d\mu(g) d\nu(x) \\
&= \Vert(\ind_H^G 1_H)(\mu)\vfi\Vert^2 .
\end{align*}
Since  $\ind_H^G1_H$ is equivalent to $\la_{G/H},$ the claim follows. $\bsq$
\end{proof}

We will also need  (in Section~\ref{S5}) a  precise description
of the kernel of an  induced representation.

\begin{lemma}
\label{Lem-KernelInduced}
 With  the notation as in the previous proposition,
 let $\pi= \ind_H^G\sigma$.
 Then   $\Ker (\pi)= \bigcap_{ g\in G} g\Ker (\sigma)g^{-1}$,
 that is,  $\Ker (\pi) $ coincides the largest 
normal subgroup of $G$ contained in $\Ker \sigma.$
\end{lemma}
\begin{proof}

Let $c: H\bs G\times G\to H$ be the cocycle corresponding
to a measurable section $s:H\bs G\to G$ with $s(H)=e.$
Let $a\in \Ker (\pi).$ Then, for every $\xi\in L^2(H\bs G, \K),$
we have
$$
\sigma(c(x,a))\xi(xa)=\xi(x) \tout   x\in H\bs G.
$$
Taking for $\xi$ mappings supported on a neighbourhood of $Ha,$
we see that  $a\in H.$ Hence $c(H,a)=a$. Taking for $\xi$  continuous mappings
with $\xi(H)\neq 0$ and evaluating at $H,$ 
we obtain that $a\in \Ker (\sigma).$
Since $\Ker (\pi)$ is normal
in $G,$ 
it follows that   $gag^{-1}\in \Ker (\sigma)$ for all
$g\in G.$  

Conversely, let $a\in G$ be such that  $gag^{-1}\in \Ker (\sigma)$ for all
$g\in G.$  Since 
$$
s(x)a= (s(x) as(x)^{-1})s(x) ,$$
we have $c(x,a)= s(x) as(x)^{-1}$  for all $x\in H\bs G.$ 
Hence, for  every $\xi\in L^2(H\bs G,\K )$  and $x\in H\bs G,$ we have
$$
(\pi(a)\xi)(x)=\sigma(c(x,a))\xi(xa)=\sigma(s(x) as(x)^{-1})\xi(x)= \xi(x).
$$
This shows that  $a\in \Ker (\pi)$ and the claim is proved. $\bsq$
\end{proof}

\section{Proof  of Theorem~\ref{Theo3} }
\label{S-ProofGenAff}
Let  $T=V/\Delta$  be a torus and $H$ a countable subgroup 
of  $\Aff(T)=\Aut (T)\ltimes T$. 
The implication $(iii) \Longrightarrow (ii)$ is obvious and 
the implication $(ii) \Longrightarrow (i)$  follows from  \cite{JuRo}. 
The fact that $(ii)$ implies $(iii)$  has been proved in  Theorem~\ref{Theo4}.
Therefore, it remains  to  show that $(i)$ implies $(ii).$
Again by Theorem~\ref{Theo4}, it suffices to show that 
if the action of $H$ on $T$  has no spectral gap,
 then  the same is true for the action of $\paut(H)$  on $T$, where $\paut$ is  the  projection 
 from  $\Aff (T)$ to $\Aut(T).$
This will be  an immediate consequence
of the next proposition.

For a probability measure $\mu$ on $\Aff(T)$,
we denote by  $ \paut(\mu)$ the probability measure on $\Aut (T)$
which is the image of $\mu$ under  
$\paut$.
Let  $U_0$ be the Koopman representation of $\Aff(T)$ on
$L^2_0(T).$

\begin{proposition}
\label{Pro-AffAut}
For every probability measure $\mu$   on $\Aff(T)$,
 we have 
$$\Vert U_0(\mu)\Vert \leq  \Vert U_0(\paut(\mu))\Vert.$$
\end{proposition}
\begin{proof}
Set $\Ga= \Aut(T).$ 
 Let $\widehat{T}\cong \ZZ^d$ be the dual group of $T.$
The Fourier transform sets up a unitary equivalence between
$U_0$ 
and the  representation $V$  of $\Aff(T)$ on
$\ell^2\left(\widehat{T}\setminus\{1_T\}\right)$ 
given by 
$$
V(\ga, a)\chi = \chi(a)\chi^\ga \tout \chi\in  \widehat{T}\setminus\{1_T\}, \ga\in \Ga, \ a\in T,\leqno(*)
$$
 where $\chi^\ga \in \widehat{T}$ is defined by $\chi^\ga(x)=\chi(\ga^{-1} (x)).$
 
 Choose a set of representatives $S$ for the $\Ga$-orbits in
$\widehat{T}\setminus\{1_T\}.$ 
Then $\ell^2\left(\widehat{T}\setminus\{1_T\}\right)$ 
decomposes as the direct sum of $\Aff(T)$-invariant subspaces
$$\ell^2\left(\widehat{T}\setminus\{1_T\}\right)=
\bigoplus_{\chi\in S}  \ell^2(\cal O_\chi),
$$
where $\cal O_\chi$ is the  
orbit of  $\chi\in S$ under $\Ga.$

It follows from Formula $(*)$ above that the restriction  $V_\chi$ 
of $V$ to  $ \ell^2(\cal O_\chi)$ is equivalent to the induced
representation $\ind_{\Ga_\chi\ltimes T} ^{\Ga\ltimes T}  \widetilde{\chi},$
where $\Ga_\chi$ is the stabilizer of $\chi$ in 
$\Ga$ and where  $ \widetilde{\chi}$ is the extension of
$\chi$ to $\Ga_\chi\ltimes T$ given  by 
$$
 \widetilde{\chi} (\ga, a) =\chi(a) \tout \ga\in \Ga_\chi, \ a\in T.
$$
The proposition  will be proved  if we can show that, for all $\chi\in S,$ we have
$$
\Vert V_\chi(\mu)\Vert \leq  \Vert V_\chi(\paut(\mu))\Vert. \leqno(**)
 $$
 Now,  the restriction of $V_\chi$ to $\Ga$ is equivalent
 to the  natural representation 
 of $\Ga$ in  $\ell^2(\cal O_\chi),$
 which is the induced representation   $\ind_{\Ga_\chi\ltimes T} ^{\Ga\ltimes T}  1_{\Ga}.$ 
 Observe that $\ind_{\Ga_\chi\ltimes T} ^{\Ga\ltimes T}  1_{\Ga}$ is equivalent
 to  $\left(\ind_{\Ga_\chi } ^{\Ga }  1_{\Ga}\right) \circ \paut.$
Hence, Inequality $(**)$ follows from Herz's majoration principle 
(Proposition~\ref{Pro-Herz}) and the proof of Theorem~\ref{Theo3} is complete.
$\bsq$
\end{proof}

The following corollary   gives a more precise
information  about the spectral structure of the Koopman representation 
associated  to the action on $T$  of a countable subgroup  of $\Aff(T).$

\begin{corollary}
\label{Cor-SpectralDec}
Let $H$ be a a countable subgroup  of $\Aff(T)$  and $\Ga=\paut(H).$
There exists  a $\Ga$-invariant  torus factor  $\overline T$ of $T$
such that the projection of $H$  in $\Aff (\overline T)$  is  an
amenable group  and which is the largest one with this property:
 every other  $\Ga$-invariant torus factor $S$
 of $T$  for which  the projection of $H$  in $\Aff (S)$ is amenable
 is a factor of $\overline T.$
 Moreover, the torus factor  $\overline T$  has the following properties:
\begin{itemize}
\item[(i)]  the projection of $\Ga$ on $\Aut (\overline T)$ is a virtually polycyclic group;
\item[(ii)]  the restriction to $L^2(\overline T)^\perp$   of the Koopman representation 
of $H$ does not weakly contain the trivial representation $1_H.$
\end{itemize}

\end{corollary}
\begin{proof}
As for the proof of Theorem~\ref{Theo3}, we
proceed by duality, using Fourier analysis and identifying
$V$ and $\Delta$ with their  dual groups.

Let $V_{\rm rat} (\Ga)$ be the subspace generated
by the union of  $\Ga$-invariant rational subspaces $W$
of $V$ for which $\Ga_W$ is amenable.
 Then $V_{\rm rat} (\Ga)$ is  a $\Ga$-invariant rational subspace
and, by Proposition~\ref{Pro-AmenableQuotient}, $\Ga_{V_{\rm rat} (\Ga)}$
is amenable. 

We claim that  the natural unitary representation of $\Ga$ on 
$\ell^2(\Delta\setminus \left(V_{\rm rat} (\Ga) \cap \Delta)\right)$ 
does not weakly contain $1_\Ga.$ Indeed,  assume by contradiction
that this is not the case. Then
there exists a $\Ga$-invariant mean $m$ on  $\Delta\setminus \left(V_{\rm rat} (\Ga) \cap \Delta)\right).$
We consider the vector space  ${\overline V}=V/V_{\rm rat} (\Ga)$
with the lattice ${\overline \Delta} =p(\Delta),$ where $p:V\to \overline V$
is the canonical projection. Then $p_*(m)$ is  a  $\Ga$-invariant mean
on ${\overline \Delta} \setminus\{0\}.$ Hence, by 
Proposition~\ref{Pro-InvMean}, there exists a   non-trivial 
$\Ga$-invariant rational  $\overline W$
subspace of $\overline V$  such that the image of
$\Ga$ in $GL(\overline W)$ is amenable. 
Then $W=p^{-1}(\overline W)$ is a $\Ga$-invariant rational subspace of $V$
 for which $\Ga_W$ is amenable. This is a contradiction since 
$V_{\rm rat} (\Ga)$ is a proper subspace of  $W.$

Let $\Ga^0= \Ga\cap \ZC(\Ga)^0.$ 
By Proposition~\ref{Pro-AmenableQuotient-Bis},  the 
eigenvalues of the restriction
of every element in  $[\Ga^0,\Ga^0]$ to $V_{\rm rat} (\Ga)$ are all of modulus 1.
Hence, by Corollary~\ref{Pro-AmenableQuotient-Discret},
the image of  $[\Ga^0,\Ga^0]$ in $GL(V_{\rm rat} (\Ga))$ is 
virtually nilpotent. It follows that  $\Ga_{V_{\rm rat} (\Ga)}$ is virtually polycyclic.$\bsq$
\end{proof}

\section{Some basic facts on Kirillov's theory and on decay  of matrix coefficients of unitary representations}
\label{S2}

We first recall some basic  facts from Kirillov's theory 
of  unitary representations of nilpotent Lie groups.

For a locally compact second countable group $G,$ the unitary dual $\widehat G$
of $G$ is the set of  classes (for unitary equivalence) of irreducible
unitary representations  of $G.$

Let $N$ be a connected and simply connected nilpotent Lie group
with Lie algebra $\mathfrak n.$
Kirillov's theory provides a parametrization of $\hN$ in terms 
of the co-adjoint orbits in the
dual space  ${\mathfrak n}^*={\rm Hom}_{\mathbf R} ({\mathfrak n},\mathbf R)$
of ${\mathfrak n}.$ We will review the basic features of this theory.

Fix $l\in {\mathfrak n}^*$. There exists a polarization $\mathfrak m$ for $l,$
that is, a Lie subalgebra $\mathfrak m$ such that $l([{\mathfrak m},{\mathfrak m}])=0$
and which is of maximal dimension; the codimension of $\mathfrak m$
is $\frac{1}{2}\dim ({\rm Ad}^* (N) l)$, where ${\rm Ad}^* (N) l$  is the orbit of $l$ under the co-adjoint representation ${\rm Ad}^*$ of $N.$ 
The induced representation  
${\rm Ind}_M^N \chi_l$ is irreducible, where $M=\exp(\mathfrak m)$ and 
$\chi_l$ is the unitary character of $M$ defined by
$$\chi_l(\exp X)=e^{2\pi i l(X)},\qquad X \in{\mathfrak m}.$$
The unitary equivalence class of ${\rm Ind}_M^N \chi_l$  only depends  on the co-adjoint orbit
${\rm Ad}^* (N) l$ of $l$.
We obtain in this way a mapping 
$$
{\mathfrak n}^*/{\rm Ad}^* (N)\to \hN,\qquad {\cal O}\mapsto \pi_{\cal O}
$$
called the Kirillov mapping, 
from the orbit space ${\mathfrak n}^*/{\rm Ad}^*(N)$ of the co-adjoint representation
to  the   unitary dual $\hN$ of $N$
The Kirillov  mapping is in fact  a bijection.
For all of this, see \cite{Kirillov} or \cite{CoGr}.

We  have to recall a few general facts
about  decay of matrix coefficients of  unitary group representations,
following \cite{HoMo} and \cite{Howe3}.

Let $(\pi,\H)$ be a unitary representation of the 
locally compact group $G.$ 
The projective kernel of $\pi$  is the normal subgroup $P_\pi$
of $G$ defined by  
$$
P_\pi=\{ g\in G\ :\ {\pi}(g) =\lambda_\pi (g) I\ \text{for some}    \ \lambda_\pi(g)\in {\mathbf C} \}.
$$
Observe  that the mapping $g\mapsto \lambda_\pi(g)$ defines a unitary 
character $\lambda_\pi$ of  $P_\pi.$
  Observe also that, for $\xi, \eta\in\H,$ the absolute value
of the matrix coefficient
$$C^{\pi}_{\xi,\eta}: g\mapsto \langle \pi(g)\xi,\eta\rangle$$
is constant on cosets modulo $P_\pi.$  
For a real number  $p$ with $1\leq p <+\infty,$ the representation $\pi$ is said to be 
strongly $L^p$  modulo $P_\pi$,
if there is dense subspace $D\subset \H.$ 
such that, for every $\xi,\eta\in D,$  the function $|C^{\pi}_{\xi,\eta}|$
belongs to $L^p(G/P_\pi).$ Observe that then $\pi$ is 
strongly $L^q$  modulo $P_\pi$ for any $q>p,$
since $C^{\pi}_{\xi,\eta}$ is bounded.

Moreover, if $\pi$ is strongly $L^2$ modulo $P_\pi,$  then $\pi$ is contained in
 an infinite multiple
of $\ind_{P_\pi}^G \lambda_\pi$ (this can be shown by a straightforward
adaptation of Proposition~1.2.3 in Chapter V of \cite{HoTa}).

We will also use the notion of a projective  representation.
   Recall that a  mapping $\pi: G \to U(\H)$
from $G$ to the unitary group   of the Hilbert space $\H$
is  a projective representation of $G$
if the following holds:
\begin{itemize}
 \item $\pi(e)=I$,
\item for all $g_1,g_2\in G,$ there exists  $c(g_1 , g_2 )\in\mathbf C $
such that  $$\pi(g_1 g_2 ) = c(g_1 , g_2 )\pi(g_1 )\pi(g_2 ),$$
\item the function $g\mapsto \langle\pi(g)\xi,\eta\rangle$ is measurable for all
$\xi,\eta\in\H.$
\end{itemize}
The mapping  $c:G \times G \to {\mathbf S}^1$ is
a $2$-cocycle
with values in the unit cercle ${\mathbf S}^1.$
The projective kernel of $\pi$ is defined
in the same way as for an ordinary representation.
Every projective unitary representation of $G$
 can be  lifted to 
an ordinary unitary representation of a central extension
of $G $ (for all this, see \cite{Mackey-Livre} or \cite{Mackey}).

\section{Decay of extensions of irreducible representations of nilpotent Lie groups }
\label{S4}
Let $N$ be a connected and simply connected nilpotent Lie group
with Lie algebra $\mathfrak n.$

The group $\Aut (N)$ of continuous automorphisms of $N$ 
can be identified with
the group $\Aut (\mathfrak n)$ of automorphisms of the Lie algebra
$\mathfrak n$  of $N,$
by means of the mapping $\varphi\mapsto d_e\varphi,$
where $d_e\varphi: \mathfrak n\to \mathfrak n$ is the differential
of $\varphi\in\Aut (N)$ at the group unit.
In this way,  $\Aut (N)$ becomes  
 an algebraic subgroup of $GL({\mathfrak n})$.
Therefore, the group $\Aff(N)=\Aut(N) \ltimes N$ of affine transformations
of $N$ is  also an algebraic group  over $\mathbf R$.

Set $G:=\Aff(N).$
In the following, we view $N$ as a normal subgroup of $G.$
The group $ G$ acts by inner automorphisms on $N$ and hence by automorphisms
on ${\mathfrak n},{\mathfrak n}^*,$
and $\hN;$ observe that, for  $g\in G$ and $l\in {\mathfrak n}^*$, we have
$$
(\Ad^*(n) l)^g= \Ad^*(gng^{-1}) (l^g)\qquad \tout n\in N.
$$
This shows that $g$ permutes the  orbits of the co-adjoint representation,
mapping the   orbit of $l$  onto the  orbit of $l^g.$
Let $\pi\in\hN$ with corresponding co-adjoint orbit ${\cal O}.$
The  representation $\pi^g\in \hN,$ defined
by $\pi^g(n)= \pi(gng^{-1}),$ corresponds to the
orbit ${\cal O}^g.$

For a co-adjoint orbit ${\cal O}$ in ${\mathfrak n}^*,$ we
denote by  $G_{\cal O}$ the stabilizer of ${\cal O}$ in   $G$.
Similarly,  
$${ G}_\pi=\{g\in G\ :\ \pi^g\ \text{is equivalent to}� \ \pi\}
$$ 
is the stabilizer in $G$ of $\pi\in\hN.$ 
Observe  that, if $\pi$ is the representation corresponding
to the co-adjoint orbit ${\cal O}$ in Kirillov's picture, then
${ G}_\pi=G_{\cal O}.$ Observe also that $N$ is contained in ${ G}_\pi$.

The following elementary  fact will be crucial for the sequel.
\begin{proposition}
\label{Prop-StabAlg}
Let $\pi$ be an irreducible unitary representation of $N$. 
The stabilizer $G_\pi$  of $\pi$ is an  algebraic subgroup
of  $G$. Moreover,  for every $l$ in the co-adjoint orbit 
 corresponding to $\pi,$  we have $G_\pi= G_lN$ where 
$G_l$ is the stabilizer of $l$ in $G$ 
\end{proposition}

\begin{proof}
  The co-adjoint orbit $\cal O$ associated to 
$\pi$  is an  algebraic
subvariety of  ${\mathfrak n}^*$ (see Theorem~3.1.4 in \cite{CoGr}).
It follows that $G_\pi= G_{\cal O}$ is an algebraic subgroup
of  $G.$
Moreover, since $N$ acts  transitively  on $\cal O$,
it is clear that ${G}_{\cal O}=G_l N$ for every $l\in \cal O.$
$\bsq$
\end{proof}

Let $\pi$ be an irreducible unitary representation of
$N$, with Hilbert space $\H.$ It is a well-known 
 part of Mackey's theory
of unitary representations of group extensions 
that there exists a projective unitary representation $\widetilde\pi$
of ${G}_{\pi}$ on $\H$ which extends $\pi$. 
Indeed, for every
$g\in {G}_\pi$,  there exists a unitary
operator $\widetilde\pi(g)$ on $\H$ such that
$$
\pi (g(n))= \widetilde\pi(g) \pi(n) \widetilde\pi(g)^{-1} \tout n\in N.
$$
One can choose  $\widetilde\pi(g)$  
such that $g\mapsto \widetilde\pi(g)$ is a projective representation
unitary representation of ${ G}_\pi$ which extends $\pi$ (see Theorem~8.2 in \cite{Mackey}).

The following proposition, which  will play a central  r\^ole in our proofs,
is a consequence of  arguments from \cite{HoMo} 
concerning decay properties of unitary representations of algebraic groups.

\begin{proposition} 
\label{Prop1} 
Let $\pi$ be an irreducible unitary representation of
$N$ on $\H$ and let $\widetilde{\pi}$
be a projective unitary representation  of ${G}_{\pi}$  which extends ${\pi}.$
There exists a real number  $p\geq 1,$ only depending  on the
dimension of ${G},$ such that $\widetilde{\pi}$  is strongly $L^p$  modulo  its projective
kernel.
\end{proposition}
\begin{proof}
Since $\pi$ is irreducible,   $\widetilde\pi(g)$ is uniquely determined  up to a
scalar multiple of the identity operator $I$  for every $g\in G_{\pi}.$
In particular, all projective unitary representations  of ${G}_{\pi}$ which extend $\pi$
have  the same projective  kernel. 

We will need to  give an explicit construction of a projective representation
of ${G}_{\pi}$ extending $\pi$. This representation will lift
to an ordinary representation of a  two-fold cover  of $G_\pi.$

We denote by  ${\cal O}$  the co-adjoint orbit associated to 
$\pi$ and  we fix throughout the proof
a linear functional  $l$ in ${\cal O}.$  

Set $H= \Aut (N)$ so that $G=H\ltimes N.$
 Let $H_l$ be the stabilizer of $l$ in $H$. As shown
in  Proposition~\ref{Prop-StabAlg}, 
$G_\pi$  is an algebraic subgroup of $G$ and $G_\pi= H_l N.$
It is clear that  $H_l$ is also an  algebraic subgroup of $G$.
Let $U_l$ be the unipotent radical of  $H_l.$ 
Then $U=U_l N$ is the   unipotent radical of $G_\pi.$

\medskip\noindent
$\bullet$\emph{First step:} We claim  that $\pi$ can be extended to an ordinary unitary representation 
$\sigma$ of $U$.

Indeed, let ${\mathfrak u}_l$ be the Lie algebra of $U_l$.
We extend $l$ to a linear functional 
$\widetilde l$ on the Lie algebra 
${\mathfrak u}= {\mathfrak u}_l\oplus {\mathfrak n} $
of $U$
by defining $\widetilde l (X)= 0$ for all $X\in{\mathfrak u}_l .$

 Let ${\mathfrak m }\subset {\mathfrak n }$ be   a polarization for $l.$
We claim that $\widetilde{{\mathfrak m}}:= {\mathfrak u}_l \oplus {\mathfrak m}$
is a polarization  for $\widetilde l.$ Indeed,
we have   $\widetilde l([\widetilde{{\mathfrak m}},\widetilde{{\mathfrak m}}])=0$
since $[X,Y]\in {\mathfrak n}  $ and $(\exp X)l = l$ for all
 $X\in {\mathfrak u}_l$ and $Y\in {\mathfrak m}.$
Moreover, the codimension of $\widetilde{{\mathfrak m}}$
in ${\mathfrak u}$  coincides with 
the codimension of ${{\mathfrak m}}$
in $\frak{n}$ and the dimension of the co-adjoint orbit 
of $\widetilde l$ under ${\rm Ad}^* (U)$ coincides with the dimension of ${\rm Ad}^* (N) l.$
Since the codimension of   $\mathfrak m$ in ${\mathfrak n}^*$
is $\frac{1}{2}\dim ({\rm Ad}^* (N) l),$ it follows that  the codimension of   $\widetilde{\mathfrak m}$ in ${\mathfrak u}^*$
is $\frac{1}{2}\dim ({\rm Ad}^* (U) \widetilde l).$ Hence,
$\widetilde{{\mathfrak m}}$
is a polarization  for $\widetilde l.$

Recall that  $\pi$ is unitarily equivalent to  the induced representation  
${\rm Ind}_M^N \chi_l$, where $M=\exp(\mathfrak m)$ and 
$\chi_l$ is the unitary character of $M$ defined by
$$\chi_l(\exp X)=e^{2\pi i l(X)} \tout  X \in{\mathfrak m}.$$
Let $\widetilde M$ be the closed subgroup  of $U$ corresponding to $\widetilde{{\mathfrak m}}$.
The unitary character $\chi_{\widetilde l}$ of 
$\widetilde M$ given by $\widetilde l$ coincides with $\chi_l$ on $M.$
Since a fundamental domain for $M\bs N$ is also a fundamental domain
for $\widetilde M \bs U,$ we see that ${\rm Ind}_{\widetilde M}^U \chi_{\widetilde l}$
can be realized on the  Hilbert space of ${\rm Ind}_M^N \chi_l$ and that 
$\sigma:= {\rm Ind}_{\widetilde M}^U \chi_{\widetilde l}$ extends $\pi= {\rm Ind}_M^N \chi_l.$


\medskip\noindent
$\bullet$\emph{Second step:} We claim that  $G_\sigma= G_\pi.$

It is obvious that $G_\sigma\subset G_\pi.$
Let  $H_l= R U_l$ be a Levi decomposition of $H_l,$
where $R$ is  a reductive subgroup of $G_l.$
In order to show that $G_\pi\subset G_\sigma,$  it suffices
to prove that $R\subset G_\sigma$, since  $ G_\pi= RU.$ 
Now,   $R$ leaves ${\mathfrak u}_l$ and ${\mathfrak n}$ invariant
and fixes $l.$  Hence, $R$  fixes
the extension  ${\widetilde l}$ of $l$ defined above
and the claim follows.

\medskip\noindent
$\bullet$\emph{Coda:}
As a result, upon replacing $N$ by $U$, we can assume that $N$
is the unipotent radical of $G_\pi.$ 
Since the connected component of $G_\pi$ has finite index, we can 
also assume that $G_\pi$ is connected. 

As shown above, we have a Levi decomposition 
$G_\pi =R N$ with $R$ a reductive subgroup contained
in $G_l.$  According to \cite{Howe2}, we can find in $N$ algebraic subgroups
$K_1\subset P_1\subset N_1$ with the following properties:
\begin{itemize}
 \item $K_1$, $P_1$, and $N_1$ are normalized by $R;$
\item $K_1$ and  $P_1$ are normal in $N_1$ and $N_1/K_1$ is a Heisenberg group
with centre $P_1/K_1;$
\item there exists a unitary character $\lambda$ of $P_1/K_1$
such that $\pi$ is equivalent to the induced representation
${\rm Ind}_{N_1}^N \pi_1$, where $\pi_1$ is the 
lift to $N_1$ of the unique irreducible representation
of the Heisenberg group $N_1/K_1$  with central character $\lambda.$
\end{itemize}
The action of $R$ on $N_1/K_1$ defines a homomorphism 
from $ R$ to the symplectic group
$Sp(N_1/P_1)$ of the vector space $N_1/P_1;$ as a result, we have a homomorphism
$\vfi: RN_1\to Sp(N_1/P_1)\ltimes (N_1/K_1).$
The representation $\pi_1$ of $N_1/K_1$ extends to a 
projective representation $\omega$ of $Sp(N_1/P_1)\ltimes (N_1/K_1),$
called the metaplectic (or oscillator, or Shale-Weil) representation; 
more precisely, there exists a two-fold cover ${\widetilde{Sp}}$ of $Sp(N_1/P_1)$
and a  unitary representation $\omega$ of $ {\widetilde{Sp}}\ltimes (N_1/K_1)$
on the Hilbert space of $\pi_1$ which extends $\pi_1.$

We can lift $\vfi$ to a homomorphism 
 $\widetilde \vfi: {\widetilde R}N_1 \to {\widetilde{Sp}}\ltimes (N_1/K_1)$ for a two-fold cover  $\widetilde R$ of $R.$
Then  $\rho: = \omega \circ {\widetilde\vfi}$ is a unitary representation 
of ${\widetilde R}N_1$ on the Hilbert space of $\pi_1$
which extends   $\pi_1.$

Set  $\widetilde\pi: = {\rm Ind}_{{\widetilde R}N_1}^{{\widetilde R}N} \rho.$
Then  $\widetilde\pi$ is a unitary representation of the two-fold cover
${\widetilde G}_\pi:={\widetilde R}N$ of $G_\pi=RN;$ moreover, $\widetilde\pi$ 
extends $\pi,$ since  $\pi$ is equivalent to  ${\rm Ind}_{N_1}^N \pi_1$,
and $\rho$ extends $\pi_1.$

Observe that ${\widetilde G}_\pi$ is in general not an algebraic 
group. Let $p: {\widetilde G}_\pi \to G_\pi$ be the covering map.
Let us say that  a connected subgroup
$H$ of   ${\widetilde G}_\pi$ is reductive
if $p(H)$ is a reductive subgroup of $G_\pi.$
We claim that ${\widetilde G}_\pi$
has no non-trivial reductive normal subgroup.
Indeed,   let $H$ be a reductive normal subgroup
of ${\widetilde G}_\pi.$  Since $G_\pi=RN$ is a Levi decomposition
of $G_\pi,$   the normal subgroup  $p(H)$ of $G_\pi$
is conjugate to a subgroup of $R$ and therefore  $p(H)\subset R.$ 
Hence, $p(H)$ centralizes $N.$
It follows that $p(H)$ is trivial since $p(H)\subset \Aut (N).$

Now, the same arguments as those on pages 87--93 in \cite{HoMo} show that there exists
an integer $k$ such that  the $k$-fold tensor
power $\widetilde \pi^{\otimes k}$ of $\pi$  is square integrable modulo the projective kernel $P_{\widetilde\pi}$ of $\widetilde\pi.$
For instance, let us check how  the first step in \cite{HoMo}
towards this claim  carries over to our situation. 
For an integer $k,$ we are interested in the tensor power $\widetilde \pi^{\otimes k}.$
In order to apply Mackey's tensor product theorem 
(see \cite[Theorem 3.6]{Mackey-Livre}),  we  have to show that 
$({\widetilde R}N_1)^k$ and the diagonal subgroup 
$\Delta {\widetilde G}_\pi$ of  ${\widetilde G}_\pi^k$ are regularly related.
Now, the quotient  space ${\widetilde G}_\pi^k/({\widetilde R}N_1)^k$ is 
can be  canonically identified with ${ G}_\pi^k/({R}N_1)^k$,
and the action of    $\Delta {\widetilde G}_\pi$ on  ${\widetilde G}_\pi^k/({\widetilde R}N_1)^k$
corresponds, via the  covering mapping $p: {\widetilde G}_\pi \to G_\pi$,
to the action of   $\Delta { G}_\pi$ on  ${ G}_\pi^k/({R}N_1)^k.$
Since  $\Delta { G}_\pi$ of  ${G}_\pi^k$ are algebraic subgroups
of $G_\pi^k,$ the claim follows. $\bsq$

\end{proof}
 \begin{remark}
  \label{Rem1}
According to \cite[p.93]{HoMo}, a crude bound for the number $p$  in Proposition~\ref{Prop1}  is 
$$
p\leq (\dim (G_{\pi}) +1)^2.
$$  
 The  generalized metaplectic representation $\widetilde{\pi}$
which appears in the proof  above has been studied by several authors
(see \cite{Duflo}, \cite{Howe2}, \cite{LionVergne}).
\end{remark}

\section{Rational unitary representations of a nilpotent Lie group}
\label{S5}

As in the previous section,
let $N$ be a connected and simply connected nilpotent Lie group
and $$G:=\Aff (N)=\Aut (N)\ltimes N.$$

Let $\pi$ be an irreducible unitary representation of $N$ and $G_\pi$ the stabilizer of $\pi$ in $G.$ Let   ${\widetilde\pi}$ be 
a  projective unitary representation  of $G_\pi$  extending $\pi.$ 
In the following proposition, we  describe  the projective kernel  $P_{\widetilde\pi}$ of $\widetilde \pi.$

 \begin{proposition}
\label{Prop-ProKerExt}
Let $L_\pi$ be the connected component  of $\Ker(\pi).$
Set ${\overline N}= N/L_\pi$ 
and let $p: N\to \overline{N}$ be the canonical projection.
For  $g=(h,n)\in G_\pi$ with $h\in \Aut(N)$ and $n\in N,$ the following are conditions are equivalent:
\begin{itemize}
\item[(i)] $g\in P_{\widetilde\pi}$;
\item[(ii)]  $h$  leaves $ L_\pi$ invariant and  the automorphism of $\overline N$ induced by $h$ 
coincides with the inner automorphism  $\Ad(p(n)^{-1})$.
\end{itemize}
\end{proposition}

\begin{proof}
Assume that  $g=(h,n)\in  P_{\widetilde\pi}$.
By definition of $P_{\widetilde\pi},$ we have 
$\widetilde \pi(h)=\la_\pi(g)\pi(n^{-1}).$
It follows that, for every  $x\in N$
$$
\pi(h(x))= \widetilde \pi(h)\pi(x)\widetilde\pi(h)^{-1} =  \pi(n^{-1})\pi(x)\pi(n)= \pi(n^{-1}xn),
$$
that is, 
$$
h(x)n^{-1}x^{-1}n \in  \Ker(\pi) \tout x\in N.
$$
Since $N$ is connected, this is equivalent to 
$$
h(x)n^{-1}x^{-1}n \in  L_\pi \tout x\in N.
$$
As $L_\pi$ is normal in $N,$ this shows that $L_\pi$ is invariant under $h$ 
and  that the automorphism induced by $h$ on $\overline N$ is $\Ad(p(n)^{-1}).$

Conversely, suppose that $L_\pi$ is invariant under $h$ 
and  that the automomorphism $\overline h$ induced by $h$ on $\overline N$ 
coincides with  
$\Ad(p(n)^{-1}).$
Observe that $\pi$ factorizes to a representation $\sigma$ of ${\overline N}.$
Let   $\widetilde{\sigma}$ be an extension of  $\sigma$ to 
the stabilizer of $\sigma$ in $\Aut(\overline N)\ltimes \overline N.$
Then
$$
 \widetilde{\sigma} ({\overline h})\sigma(p(x))\widetilde{\sigma} ({\overline h})^{-1}  = \sigma(p(n))^{-1}\sigma(p(x))  \sigma(p(n))
 \tout  x \in N,
$$
that is, $\sigma(p(n)) \widetilde{\sigma} ({\overline h})$ commutes with  $\sigma(p(x))$ for all $x\in N.$
Since $\pi$ is irreducible, it follows that  $ \sigma(p(n))\widetilde{\sigma} ({\overline h})$
and hence   $ \pi({n})\widetilde{\pi} ({h})$ is  a scalar operator. This means that  $g=(h,n)\in P_{\widetilde\pi}. \bsq$ 
\end{proof}

Next, we  review some well-known facts about rational structures on $\mathfrak n$
(see \cite{CoGr}, \cite{Raghunathan}). 

Recall first that a lattice $\Ga$ in a locally compact  group $G$
 is a discrete subgroup such that the  translation invariant measure 
induced by a Haar measure on $G$
on  the homogeneous space $\Ga\backslash G$ is finite.

The Lie algebra $\mathfrak n$ (or
the corresponding nilpotent Lie group $N=\exp ({\mathfrak n})$) has
a \emph{rational structure} if there is a Lie algebra ${\mathfrak n} _{\mathbf Q}$
over
 $\mathbf Q$ such that ${\mathfrak n}\cong {\mathfrak n}_{\mathbf Q}\otimes_{\mathbf Q} \mathbf R.$
If  $\mathfrak n$ has a rational structure given by  ${\mathfrak n}_{\mathbf Q}$,
then $N$ contains a  cocompact lattice $\La$ 
such that $\log \La \subset {\mathfrak n}_{\mathbf Q}$. 
Conversely,   if $N$  contains a  lattice $\La$,
then $\La$ is cocompact and $\mathfrak n$ has a rational structure  
given by ${\mathfrak n}_{\mathbf Q}={\mathbf Q}-{\rm span}  (\log \La).$ 

Assume from now on that $N$ has a 
rational structure  ${\mathfrak n}_{\mathbf Q}$ and let 
$\La$ be a lattice inducing this rational structure.
We say that a $\mathbf R$-subspace $\mathfrak h$  of $\mathfrak n$ is   rational if
$\mathfrak h={\mathbf R}-{\rm span} ({\mathfrak h}\cap{\mathfrak n}_{\mathbf Q}).$
All subalgebras in the ascending or ascending series
as well as the centre of $\mathfrak n$ are rational.
A connected closed subgroup $H$ of $N$ is said to be rational
if the corresponding subalgebra Lie algebra $\mathfrak h$   is rational.
This is equivalent to the fact that
$H\cap \La$ is a  lattice in $H.$

Let $H$ be a rational connected normal closed subgroup of $N$
with Lie algebra ${\mathfrak h}$
Then $N/H$ has a canonical rational structure $({\mathfrak n}/{\mathfrak h})_{\mathbf Q}$
induced  by the lattice $\La H/H$ of $N/H.$

 There is a unique  rational structure ${\mathfrak n}^*_{\mathbf Q}$ on 
the dual space ${\mathfrak n}^*$
 defined as follows: a functional $l\in {\mathfrak n}^*$ belongs to ${\mathfrak n}^*_{\mathbf Q}$ if and only
if $l(X)\in \mathbf Q$ for all $X\in  {\mathfrak n}_{\mathbf Q}.$

\medskip

 An important role will be played later (in Section~\ref{S7})
 by  irreducible unitary representations of $N$  which are rational
in the sense of the following definition.
\begin{definition}
 \label{Def2}
An irreducible  unitary representation
$\pi$ of $N$ is \emph{rational} if its co-adjoint orbit ${\cal O}_\pi$
is rational, that is, if ${\cal O}_\pi\cap {\mathfrak n}^*_{\mathbf Q}\neq \emptyset.$
\end{definition}

We fix for the rest of this section 
 a rational  irreducible unitary representation $\pi$  of
$N.$

We first  establish  the rationality of 
the kernel  of $\pi$.
\begin{proposition}
 \label{Prop-RatStab}
The  connected component  $L_\pi$ of $\Ker(\pi) $  is  a rational normal subgroup of $N.$
As a consequence, $\overline \La= \La L_\pi/L_\pi$ is a lattice in $N/L_\pi.$
\end{proposition}
\begin{proof}
Since $\pi$ is rational,
the corresponding co-adjoint orbit  in 
${\mathfrak n}^*$  
contains a  functional $l\in {\mathfrak n}^*_{\mathbf Q}.$
The representation $\pi$ is unitarily equivalent
to   ${\rm Ind}_M^G \chi_l$, where 
$\mathfrak m$ is a polarization for $l,$ 
$M=\exp(\mathfrak m),$ and 
$\chi_l$ is the unitary character of $M$ corresponding to $l.$

Recall from Lemma~\ref{Lem-KernelInduced} that 
$\Ker (\pi)$ coincides with the largest normal subgroup of $N$ contained in 
$\Ker(\chi_l).$ 
For the ideal  $\mathfrak l$ corresponding to $\Ker(\pi),$ we have therefore
$$
{\mathfrak l} = \bigcap_{n\in N} \Ker (\Ad^*(n) l)=
\bigcap_{X\in {\mathfrak n}_{\mathbf Q} }\Ker (\Ad^*(\exp X) l).
$$
Since $\Ker (\Ad^*(\exp X ) l)$ is rational for all $X\in{\mathfrak n}_{\mathbf Q},$
it follows that ${\mathfrak l}$ is rational.
Thus,
the connected component  $L_\pi$ of $\Ker(\pi)$ is rational, by definition. $\bsq$
\end{proof}

\medskip

The  set $\Aut (\nil)$ consisting 
of the automorphisms $\ga\in \Aut (N)$ with $\ga(\La)=\La$ is 
a discrete subgroup of the algebraic group $\Aut (N).$

Let $G_\pi$ be the stabilizer of  $\pi$ in $G$  and $\widetilde \pi$  a projective 
unitary representation of $G_\pi$ extending $\pi$.
Set 
$$\Ga_\pi=G_\pi\cap \Autnil.$$

The projective kernel $P_{\widetilde\pi}$  of $\widetilde \pi$ was determined in 
Proposition~\ref{Prop-ProKerExt} . We will need
to have  a precise  description of  $P_{\widetilde\pi}\cap (\Ga_\pi \ltimes N). $

 As  before, let $L_\pi$ be the connected component  of $\Ker(\pi),$
 $\overline N= N/L_\pi,$ $p:N\to \overline N$ the canonical projection, 
 and $\overline{\La}= p(\La).$
  Observe that 
$g(L_\pi)=L_\pi$ for all  $g\in  G_\pi\cap \Aut(N).$ 
Consider the  induced   continuous homomorphism 
$$\vfi: G_\pi \to  \Aff(\overline N)=\Aut(\overline{N})\ltimes \overline N.$$ 

\begin{proposition} 
\label{Pro-ProjKernNilman}
Let ${\rm Norm}(\overline \La)$ be the normalizer 
   of  $\overline \La$ in $\overline N.$ 
\begin{itemize}
  \item[(i)]   We have 
$$P_{\widetilde\pi}\cap (\Ga_\pi \ltimes N) = 
\vfi^{-1}\left(\{(\Ad(x), x^{-1}) \ :\ x\in  {\rm Norm}(\overline \La)\}\right).$$
 \item[(ii)]  Let $\Delta:= \{(\Ad(x), x^{-1} z)\ :  x\in \overline \La, z\in Z(\overline N)  \}$,
 where  $Z(\overline N)$ is the centre of $\overline N.$
 Then  $\vfi^{-1} (\Delta) \cap (\Ga_\pi\ltimes N ) $ 
 is a subgroup of finite index in   $P_{\widetilde \pi}\cap (\Ga_\pi\ltimes N ).$ 
 \end{itemize}
 \end{proposition}
\begin{proof}
(i) By Proposition~\ref{Prop-ProKerExt},  we have
$$
P_{\widetilde\pi}= \vfi^{-1} \left(\{(\Ad(x), x^{-1})\ :  x\in \overline N  \}\right). 
$$
Let   $g=(\ga,n)\in P_{\widetilde\pi}  \cap (\Ga_\pi\ltimes N ).$
Then $\vfi(g)=(\Ad(x), x^{-1})$ for some
$x\in \overline N.$
Since $\ga(\La)=\La,$ we have $\Ad(x)(\overline \La)=\overline \La$,
that  is, $x\in  {\rm Norm}(\overline \La).$ 
Conversely, it is obvious that, if $g=(\Ad(x), x^{-1})$ for  some
$x\in  {\rm Norm}(\overline \La)$, then $g\in P_{\widetilde\pi}\cap (\Ga_\pi \ltimes N).$

  (ii) In view of (i), it suffices to prove that the subgroup $\overline  \La Z(\overline N) $ has finite index
  in   $ {\rm Norm}(\overline \La).$ 
  
  To show  this, recall   that $\overline{\La}$ is a cocompact lattice in $\overline N$   
  (Proposition~\ref{Prop-RatStab}).
Let  $ {\rm Norm}(\overline \La)_0$ be the connected  component of 
  $ {\rm Norm}(\overline \La).$ Since  $ {\rm Norm}(\overline \La)_0$ normalizes  $\overline \La$ and since
   $\overline \La$  is discrete,  $ {\rm Norm}(\overline \La)_0$ lies in the centralizer of every element of
   $\overline \La$. As  $\overline \La$ is Zariski dense in $\overline N$
   (see e.g. Theorem 2.1 in  \cite{Raghunathan}), 
   it follows that   $ {\rm Norm}(\overline \La)_0=Z(\overline N).$ 
   Since the projection  of  $\overline \La$ has finite covolume
   in the discrete group  $ {\rm Norm}(\overline \La)/ {\rm Norm}(\overline \La)_0,$ the claim follows.$ \bsq$

\end{proof}

The next proposition will allow us  to deduce
decay properties of representations of $G_\pi$ restricted to  $\Ga_\pi\ltimes N$.
\begin{proposition}
\label{Prop-ClosedSubg}
 The subgroup $(\Ga_\pi\ltimes N) P_{\widetilde\pi} $ is closed in $G_\pi.$
 \end{proposition}

\begin{proof}
Using Proposition~\ref{Prop-ProKerExt}, we see that 
$$P_{\widetilde\pi} N = \vfi^{-1} \left(\Ad(\overline N) \ltimes \overline N\right)$$
 and hence  
 $$(\Ga_\pi\ltimes N)P_{\widetilde\pi}= \vfi^{-1} \left((\vfi(\Ga_\pi)\Ad(\overline N))\ltimes \overline N\right) .$$
 It therefore suffices  to show that  $\vfi(\Ga_\pi)\Ad(\overline N)$
 is closed in  $\Aut(\overline N).$

Observe that,  for every $\ga\in \Ga_\pi,$ we have $\ga (\La )=\La$
(since $\Ga_\pi\subset \Autnil$) and hence
 $\vfi(\Ga_\pi)\subset  \Aut(\overline{\La}\bs\overline{N}).$
 
  Let $(\ga_i)_i$ and $(x_i)_i$ be sequences in
 $\Ga_\pi$ and in $\Ad(\overline{N})$ such that  
$$\lim_i \vfi(\ga_i) x_i =g\in \Aut(\overline N).$$
 Since  $\Ad(\overline\La)$ is a cocompact lattice in  $\Ad(\overline N),$ 
 there exists a compact subset $D$ of  $\Ad(\overline N)$ such that  $x_i= \delta_i d_i$ for some
 $\delta_i\in \Ad(\overline\La)$ and $d_i\in D$.
 As $D$ is compact, we can assume that  
$\lim d_i=d\in \Ad(\overline N)$
 exists.  Then  $\lim_i \vfi(\ga_i) \delta_i= gd^{-1}.$  Now,
 $$\Ad(\overline\La)=\vfi(\Ad(\La)) \subset  \vfi(\Ga_\pi)$$
 and $\vfi(\Ga_\pi)$ is a subgroup of the
 discrete  group  $ \Aut(\overline{\La}\bs\overline{N}).$ It follows that
 $gd^{-1}\in  \vfi(\Ga_\pi),$ that is, $g\in  \vfi(\Ga_\pi)\Ad(\overline N).$
 Hence,  $\vfi(\Ga_\pi)\Ad(\overline N)$
 is closed in  $\overline N.\ \bsq$
 \end{proof}

 \begin{corollary}
\label{Cor-DecayRatRep}
 Let $\Delta= \{(\Ad(x), x^{-1} z)\ :  x\in \overline \La, z\in Z(\overline N)  \}$
 and $\vfi: G_\pi \to  \Aff(\overline N)$ the canonical projection, where $\overline N= N/L_\pi.$
  The restriction of  $\widetilde{\pi}$ to $\Ga_\pi\ltimes N$
is  strongly $L^p$  modulo   $\vfi^{-1} (\Delta) \cap (\Ga_\pi\ltimes N )$
 for the real number  $p$  appearing in Proposition~\ref{Prop1}.
\end{corollary}

\begin{proof} 
We know from Proposition~\ref{Pro-ProjKernNilman} that   $\vfi^{-1} (\Delta) \cap (\Ga_\pi\ltimes N )$ has finite index
in  $ P_{\widetilde \pi}\cap (\Ga_\pi\ltimes N ).$ Hence,
it suffices to prove that  the restriction of  $\widetilde{\pi}$ to $\Ga_\pi\ltimes N$
is   strongly $L^p$  modulo  $ P_{\widetilde \pi}\cap(\Ga_\pi\ltimes N ).$ 

By  Proposition~\ref{Prop-ClosedSubg}, $(\Ga_\pi\ltimes N)P_{\widetilde \pi}$ is closed in $G_\pi.$   Therefore,
$(\Ga_\pi \ltimes N) P_{\widetilde \pi}/P_{\widetilde \pi}$ is homeomorphic as a $(\Ga_\pi\ltimes N)$-space to 
$(\Ga_\pi \ltimes N)/ (P_{\widetilde \pi}\cap(\Ga_\pi\ltimes N)).$
It follows from Proposition ~\ref{Prop1} (see the proof of Proposition 6.2 in \cite{HoMo})  that
the restriction of  $\widetilde{\pi}$ to $\Ga_\pi\ltimes N$ 
is   strongly $L^p$  modulo  $P_{\widetilde \pi}\cap(\Ga_\pi\ltimes N).\ \bsq$
\end{proof}

\section{A general  estimate for norms of convolution operators }
\label{S-Nevo}

Let $G$ be a locally compact group.
   For a unitary representation $(\pi,\H)$ of $G$, the contragredient (or conjugate) representation
 $\overline\pi$ acts on the conjugate Hilbert space $\overline \H$.
Recall that, for an integer $k\geq 1,$   the $k$-fold tensor product $\pi^{\otimes k}$
of $\pi$  is a unitary representation of $G$ 
acting on the tensor product Hilbert space ${\H}^{\otimes k}.$ 

We will need in a crucial way the following estimate which
 appears in the  proof of Theorem~1 in \cite{Nevo}.
\begin{proposition}
\label{Pro-Nevo}
 Let $\mu$  be a  probability measure  on 
the Borel subsets of $G.$
 Let  $(\pi,\H)$ be a unitary representation  of $G$.
For every integer  $k\geq 1,$ we have 
$$
\Vert \pi(\mu)\Vert \leq \Vert \left( \pi\otimes\overline{\pi}\right)^{\otimes k}(\mu)\Vert^{1/2k},
$$
\end{proposition}
\begin{proof}
Denote by $\check{\mu}$ the probability measure on $G$  defined
by  $\check{\mu}(A)= \mu(A^{-1})$ for every Borel subset $A$ of  $G.$

 Using Jensen's inequality, we have  for every   vector $\xi\in\H,$
\begin{align*}
\Vert\pi(\mu)\xi\Vert^{4k} 
&=\vert\langle\pi(\check{\mu}\ast\mu)\xi,\xi \rangle\vert^{2k}\\
&=\left\vert\int_{ G} \langle \pi(g)\xi,\xi\rangle
d(\check{\mu}\ast\mu)(g)\right\vert^{2k}\\
&\leq\int_G\langle \vert\pi(g)\xi,\xi \rangle\vert^{2k}
d(\check{\mu}\ast\mu)(g)\\
&=\int_G\vert\langle (\pi\otimes\overline{\pi})(g)(\xi\otimes \xi),\xi\otimes \xi\rangle\vert^k d(\check{\mu}\ast\mu)(g)\\
&=\int_G\langle (\pi\otimes\overline{\pi})^{\otimes k}(g)(\xi\otimes \xi)^{\otimes k},(\xi\otimes \xi)^{\otimes k} \rangle d(\check{\mu}\ast\mu)(g)\\
&=\vert\langle(\pi\otimes\overline{\pi})^{\otimes k}(\check{\mu}\ast\mu)(\xi\otimes \xi)^{\otimes k}, (\xi\otimes \xi)^{\otimes k}\rangle\vert\\
&= \Vert (\pi\otimes\overline{\pi})^{\otimes k}(\mu)(\xi\otimes \xi)^{\otimes k}\Vert^{2}.
\end{align*}
and the claim follows. $\bsq$
\end{proof}

\section{Analysis of the Koopman representation of the affine group of a nilmanifold}
\label{S7}

Let $N$ be a connected and simply connected nilpotent Lie group,
$\La$  a lattice in $N.$ 
There is a unique  translation invariant probability
measure $\nu_{\nil}$ on  $\nil$ and it is  induced by a Haar measure
on $N.$ This measure is also invariant under $\Aut (\nil).$

We fix throughout this section   a subgroup $\Ga$ of  $\Autnil.$ 
The Koopman representation $U$ of $\Ga\ltimes N$ 
associated to the action 
of  $\Ga\ltimes N$ on  $\nil$ is  given by
$$
U(\ga,n) \xi(x)=\xi(\ga^{-1}(x)n) \qquad\ga\in\Ga,\  n\in N, \ \xi\in L^2(\nil), \ x\in \nil.
$$
In particular, we have
$$
 U({\ga^{-1}}) U(n) U({\ga})= U({\ga^{-1}(n)}) \tout \ga\in \Ga,\ n\in N.\leqno{(1)}
$$

Recall that $T=\tor$ is the maximal factor torus associated 
to $\nil.$ The action of $\Affnil$ on $\nil$ induces an action
of  $\Affnil$ on $T.$ 
We identify $L^2(T)$ with a closed subspace
of  $L^2(\La\bs N).$

 More generally, let $L$ be a 
connected closed  
subgroup of $N$ which is both rational and invariant under $\Ga.$
Then $\La \cap L$ is a lattice in 
$L$ and $\overline \La= \La L/L$ is a lattice
in $\overline N= N/L.$   There is an induced
action  of  $\Ga\ltimes N$ on the subnilmanifold
$L/(\La\cap L)$ and on the factor nilmanifold
$\overline \La \bs \overline N.$ 
The canonical mapping $p:\nil \mapsto {\overline \La}\bs {\overline N}$
is $\Ga\ltimes N$-equivariant and presents $\nil$ as
a fibre bundle over  ${\overline \La}\bs {\overline N}$
with fibres diffeomorphic to  $L/(\La\cap L).$ 
The Hilbert space $L^2(\overline \La\bs \overline N)$ can be identified,
as $\Ga\ltimes N$-representation, with the  $\Ga\ltimes N$-invariant closed subspace of 
$L^2(\nil)$ consisting of the  square-integrable functions on $\nil$ 
which are constant on the fibres of $p.$

We write $$L^2(\La\bs N)= L^2(T) \oplus \H,$$ where
$\H$ is the  orthogonal complement of  $L^2(T)$
on $L^2(\nil),$ and
observe that $\H$
is invariant under $\Affnil.$

We are going to show that  the restriction 
of $U$ to $\H$ has a canonical decomposition  
into a direct sum of induced
representations from  the stabilizers in $\Ga\ltimes N$ of 
certain representations $\pi\in \hN$;
this decomposition can be viewed
as  generalization of the decomposition 
of $L^2(T)$ which appears in the proof of Proposition~\ref{Pro-AffAut}.

 Since $\La$ is cocompact in $N,$ we   can consider the decomposition of 
 $\H$  into its $N$- isotypical components:  we  have
$$
\H=\bigoplus_{\pi\in\Sigma} \H_{\pi},
$$
where $\Sigma$ is a certain set of infinite-dimensional
pairwise non-equivalent irreducible unitary representations of $N;$
for every $\pi\in\Sigma$, the space $\H_{\pi}$
is the union of the closed $U(N)$-invariant subspaces $\K$
of $\H$ for which the corresponding representation of 
$N$ in ${\K}$   is  equivalent to $\pi$.
According to \cite[Corollary2]{Moore}, every $\pi\in\Sigma$
is rational in the sense of Section~\ref{S5}.
Every $\H_{\pi}$ is a direct sum of finitely many irreducible unitary  representations;
 therefore,  the restriction  of $U(N)$  to ${\H_{\pi}}$ is unitarily equivalent
to a tensor product $\pi \otimes I$
acting on $\K_\pi\otimes \L_\pi,$
where  $\K_{\pi}$ is the Hilbert space of 
$\pi$ and where $\L_{\pi}$ is a finite
dimensional Hilbert space.  (For a precise computation of the dimension of $\L_\pi$,
see \cite{Howe1} and \cite{Richardson}; the fact that $\L_\pi$
is finite-dimensional will not be relevant for our arguments.)

Let $\ga$ be a fixed automorphism in $\Ga.$ 
Let $U^\ga$ be the conjugate representation of $U$ by $\ga,$ that is,
$ U^\ga(g)=U(\ga^{-1}(g))$ for all $g\in G.$
On the one hand, for every $\pi\in\Sigma,$ 
 the subspace $\H_{\pi^{\ga^{-1}}}$
 is the isotypical component of  $U^\ga\vert _{N}$
 corresponding to $\pi.$
On the other hand,  relation $(1)$ shows that
$U(\ga^{-1})$ provides a unitary equivalence between
 $U\vert_{N} $ and  $U^\ga\vert _{N}.$ 
  It follows that 
 $$
U({\ga}^{-1}) (\H_{\pi}) = \H_{\pi^{\ga{-1}}}\tout \ga\in \Ga
$$
In summary, we see that
  $\Ga$ permutes the $\H_\pi$'s among themselves
according to its action on $\hN.$

Write $\Sigma=\bigcup_{i\in I} \Sigma_i,$ where the  $\Sigma_i$'s are   the $\Ga$-orbits
 in $\Sigma,$ and set
$$
\H_{\Sigma_i}=\bigoplus_{\pi\in\Sigma_i} \H_{\pi}.
$$
Every $\H_{\Sigma_i}$ is invariant under $\Ga_i\ltimes N$
and we have  an orthogonal decomposition  
$$
\H= \bigoplus_{i} \H_{\Sigma_i}.
$$
Fix $i\in I.$ Choose  a representation
${\pi}_i $ in $\Sigma_i$ and set $\H_i= \H_{\pi_i}.$
Let  $\Ga_i$  denote the stabilizer
of $\pi_i$  in $\Ga.$ 
The space  $\H_i$
is invariant under  $\Ga_i\ltimes N.$ 
Let $V_i$ be  the corresponding representation
of $ \Ga_i\ltimes N$  on $\H_i.$

Choose   a set $S_i$ of representatives
for the cosets in 
$$\Ga/\Ga_i= (\Ga\ltimes N)/ (\Ga_i\ltimes N)$$
 with $e\in S_i.$
Then $\Sigma_i=\{ \pi_i^s\ :\ s\in S_i\}$ 
and the Hilbert space $\H_{\Sigma_i}$
is the sum of mutually orthogonal spaces:
$$
\H_{\Sigma_i}=
\bigoplus_{s\in S_i}\H_i^s.
$$
Moreover, $\H_i^s$
is the image under $U (s)$
of $\H_i$
for every $s\in S_i.$ This exactly means  that  the restriction   $U_i$ of $U$
to  $\H_{\Sigma_i}$ of the Koopman representation $U$  of $\Ga\ltimes N$  
is equivalent to the induced representation 
$\ind_{\Ga_i\ltimes N}^{\Ga\ltimes N} {V_i}.$

As we have seen above, we can assume that
$\H_i$ is the tensor product 
$$
\H_i  =\K_i\otimes \L_i
$$
of the Hilbert space $\K_i$ of $\pi_i$ with
a finite dimensional Hilbert space $\L_i,$ 
in such a way that
$$
V_i(n)= \pi_{i}(n) \otimes I_{\L_i} \tout n\in N.\leqno{(2)}
$$
Let  $g\in \Ga_i\ltimes N.$ By $(1)$ and $(2)$ above,  we have 
 $$
 V_i(g) \left(\pi_{i}(n) \otimes I_{\L_i}\right)V_i(g) ^{-1}  = \pi_{i}(gng^{-1}) \otimes I_{\L_i} \tout  \ n\in N.\leqno{(3)}
$$
On the other hand, let  $G_i$ be the stabilizer of $\pi_i$ in $\Aff (N)$; then $\pi_i$ 
extends to   an irreducible projective representation $ \widetilde{\pi}_i$
 of $G_i$  (see the remark 
 just before Proposition~\ref{Prop1}).
Since 
$$
 \widetilde{\pi_i}(g)  \pi_{i}(n)\ \widetilde{\pi_i}(g^{-1})= \pi_{i}(gng^{-1}) \tout  n\in N, 
$$
it follows from $(3)$  that 
the operator  $\left(\widetilde{\pi_i}(g^{-1})\otimes I_{\L_i}\right)V_i(g)$
commutes with $\pi_i(n)\otimes I_{\L_i}$ for all $n\in N.$ 
 Since $\pi_i$ is irreducible,  there exists a unitary operator
$W_i(g)$ on $\L_i$ such that 
$$
V_i(g)= \widetilde{\pi_i}(g)\otimes W_i(g).
$$
 It is clear that $W_i$ is a projective unitary representation
of $\Ga_i\ltimes N$,  since  $V_i$  is a unitary representation
of $\Ga_i\ltimes N$.

\section{Proof of Theorem~\ref{Theo1}: first step}
\label{S8}

We summarize  the  discussion from the previous section.
We have a first orthogonal decomposition into $\Affnil$-invariant 
subspaces
$$
L^2(\La\bs N)= L^2(T) \oplus \H, 
$$
where $T$ is the maximal torus factor of $\nil.$
Let $\Ga$ be a  subgroup of  $\Autnil$.
There exists  a sequence
of $\Ga$-invariant  sets   $({\Sigma_i})_{i\in I}  $
of rational infinite dimensional unitary  irreducible representations of $N$ such that  we have a decomposition
into mutually orthogonal $\Ga\ltimes N$-invariant subspaces
$$\H=\bigoplus_{i\in I} \H_{\Sigma_i}
$$
with the following property:  for every $i,$  the representation  $U_i$ of 
$\Ga\ltimes N$ defined   on  $\H_{\Sigma_i}$
is equivalent to
$$
 \ind_{\Ga_i\ltimes N}^{\Ga\ltimes N} \left(\widetilde{\pi_i}\otimes W_i\right),
 $$
where  $\pi_i$ is a  representation from ${\Sigma_i},$
where $\widetilde{\pi_i}$ is  the restriction to $\Ga_i\ltimes N$
of  an extension of $\pi_i$ to  the stabilizer  $G_i$ of $\pi_i$ in $G=\Aff (N)$,
and where  $W_i$ is some finite dimensional
projective unitary representation of $\Ga_i\ltimes N.$

We need to recall the decomposition of the representation  $\Utor$ of $\Ga$ on
$L^2_0(T)$ from Section~\ref{S-ProofGenAff}.
Let $\widehat{T}\cong \ZZ^d$ be the dual group of $T$
and let $S$ be a set of representatives for the $\Ga$-orbits in
$\widehat{T}\setminus\{1_T\}.$
Then
$$
\Utor \cong \bigoplus_{\chi\in S} \la_{\Ga/\Ga_\chi}, \leqno{(4)}
$$
where $\Ga_\chi$ is the stabilizer of $\chi$ in $\Ga$ and 
$\la_{\Ga/\Ga_\chi}$ is the natural representation
of $\Ga$ on $\ell^2(\Ga/\Ga_\chi).$

In the  following result,  we establish
a link between the restrictions  to  $\H$ and  to $L^2_0(T)$ of the Koopman representation
of $\Ga$.  This result,
which is a  consequence of the discussion above  and of 
results from Section~\ref{S5}, is a major  step
in our proof of Theorem~\ref{Theo1}.

Recall that  $\paut$  denotes the 
canonical projection $\Affnil\to \Autnil$. For a probability measure
$\mu$  on $\Affnil,$  let $\paut(\mu)$ 
be  the probability measure on $\Autnil$ which is the image of  $\mu$ under $\paut.$

\begin{proposition}
\label{Pro-FixAut}
There exists an integer $k\geq 1 $ only depending on $\dim N$
with the following property.  
Let $\Ga$ be a subgroup of  $\Autnil$ which
stabillizes some $\pi\in \hN$ appearing in the decomposition
$\H=\bigoplus_{\pi\in\Sigma} \H_{\pi} $ 
of $\H$ into isotypical components under $N$. For every 
probability  measure $\mu$ on
$\Ga\ltimes N$,  we have
$$
\Vert U_\pi(\mu))\Vert \leq  \Vert \Utor(\paut(\mu))\Vert^{1/2k},
$$
where  $U_\pi$ and $\Utor$ are the restrictions of the Koopman
representation  of $\Ga\ltimes N$ to $\H_{\pi}$
and $L^2_0(T)$ respectively.
\end{proposition}

\begin{proof}
Let $G_\pi$ be the stabilizer  of $\pi$ in $G=\Aff(N).$
Let $\widetilde{\pi}$ a projective representation
of $G_\pi$ extending  $\pi.$ 

As we have seen above,  
 $U_\pi$  is equivalent  to  
$(\widetilde{\pi}\vert_{\Ga\ltimes N})\otimes W$
for  some finite dimensional 
projective unitary representation $W$ of $\Ga\ltimes N.$
Let $P$ denote the projective kernel of $U_\pi.$
Observe that  $P=P_1\cap  P_2$,
where $P_1$ and $P_2$ are the projective kernels
of  $\widetilde{\pi}\vert_{\Ga\ltimes N}$   and $W$.

Denote by $L_\pi$ the connected component of $\Ker(\pi)$ and $\overline N=N/L_\pi.$
As in  Section~\ref{S5}, let $\vfi: G_\pi\to  \Aff(\overline N)$ be the 
corresponding homomorphism 
and $$\Delta= \{(\Ad(x), x^{-1} z)\ :  x\in \overline \La, z\in Z(\overline N)  \},$$
 where  $\overline \La$ is the lattice $ \La L_\pi/L_\pi$ 
 in  $\overline N$ and $Z(\overline N)$  the centre of $\overline N.$
Then 
$$Q:=\vfi^{-1} (\Delta) \cap (\Ga_\pi\ltimes N )$$
 is  a subgroup of finite index
of  $P_1$ (Proposition~\ref{Pro-ProjKernNilman}).  
By  Corollary~\ref{Cor-DecayRatRep}, 
there exists a real number  $p\geq 1$ only depending  on the
dimension of $\Aut(N)\ltimes N$ such that 
$\widetilde{\pi}_{\vert \Ga_\pi\ltimes N}$   is strongly $L^p$  modulo  $Q$.
  
  We claim that  $Q$ is   contained in  $P.$ Indeed, 
  for $g\in Q$, we have  
  $$\vfi(g) =(\Ad(x), x^{-1} z)$$
   for some  
  $x \in \overline \La$  and  $z\in Z(\overline N) .$
  Hence $\vfi(g)$ acts as the right translation by $z$ on 
   $L^2(\overline \La\bs \overline N).$  
 Observe that $\H_\pi$ is contained in $L^2(\overline \La\bs \overline N)$
and that $g$ acts  as $\vfi(g)$ on  $\H_\pi.$
 Since $ N$ acts as a multiple of  the irreducible representation $\pi$
 on $\H_\pi$,  it follows that  $g\in P$ and the claim is proved
 
 As a consequence, we see that $Q$ is  a subgroup of finite index 
 in $P$. Observe that $Q$ is also contained in $P_2.$
 It follows that $U_\pi= (\widetilde{\pi}\vert_{\Ga\ltimes N})\otimes W$
 is strongly  $L^p$  modulo     $Q$ and hence 
 $U_\pi$ is strongly $L^p$ modulo $P$.
 
 Let $k$ be an integer with $k\geq p/4.$ Then the tensor power 
$\left( U_\pi\otimes\overline{U_\pi}\right)^{\otimes k}$ is strongly
 $L^2$ modulo $P.$
 Hence,  as discussed in Section~\ref{S2}, 
 $\left( U_\pi\otimes\overline{U_\pi}\right)^{\otimes k}$ is contained in  
 an infinite multiple
of the induced representation $\ind_{P}^{\Ga \ltimes N} \lambda_\pi, $ 
for  the associated unitary  character  $\lambda_\pi$ of $P.$  
It follows that, for every  probability measure  $\mu$ on $\Ga\ltimes N,$ we have
$$
\Vert \left( U_\pi\otimes\overline{U_\pi}\right)^{\otimes k}(\mu)\Vert \leq \Vert\left( \ind_{P}^{\Ga \ltimes N} \lambda_\pi\right) (\mu)\Vert 
$$
and hence, using  Proposition~\ref{Pro-Nevo},
$$
\Vert U_\pi(\mu)\Vert  \leq  \Vert  \left(\ind_{P}^{\Ga \ltimes N} \lambda_\pi\right)(\mu) \Vert ^{1/2k}.
$$

On the other hand,  observe that $P N=\paut^{-1}(\paut (P))$  is  closed in $\Affnil,$ 
as $\Autnil$ is discrete.
Since, by induction by stages,
$$
\ind_{P}^{\Ga \ltimes N} \lambda_\pi = \ind_{P N}^{\Ga \ltimes N} \left(\ind_{P}^{P N} \lambda_\pi\right),
$$
we have, using  by Herz's  majoration principle (Proposition~\ref{Pro-Herz}),
$$
\Vert  \left(\ind_{P}^{\Ga \ltimes N} \lambda_\pi\right)(\mu)\Vert 
\leq  \Vert  \la_{(\Ga\ltimes N)/ P N}(\mu)\Vert.
$$
Now, $ \la_{(\Ga\ltimes N)/P N} = \left(\la_{\Ga/\paut(P)}\right)\circ \paut$ and hence
$$  
\Vert  \la_{(\Ga\ltimes N)/ P N}(\mu)\Vert = \Vert\la_{\Ga/\paut(P)} (\paut (\mu))\Vert.
$$
As a consequence, the proposition will be proved if we establish the following inequality
$$
\Vert\la_{\Ga/\paut(P)} (\paut (\mu))\Vert  \leq   \Vert \Utor(\paut(\mu))\Vert. \leqno {(5)}
$$
To show this, recall (see $(4)$ above) that  $\Utor$ is equivalent
to  the direct sum $\bigoplus_{\chi\in S} \la_{\Ga/\Ga_\chi},$
where  $S$ is   set of representatives  for the $\Ga$-orbits in
$\widehat{T}\setminus\{1_T\}.$ 
As a consequence, Inequality $(5)$ will be proved
if we can show that  there exists $\chi\in \widehat{T}\setminus\{1_T\}$
such that 
$$
\Vert\la_{\Ga/\paut(P)} (\paut (\mu))\Vert \leq  \Vert \la_{\Ga/\Ga_\chi} (\mu)\Vert.
$$

By Herz's majoration principle again, it suffices to show that 
exists $\chi\in \widehat{T}$ with $\chi\neq 1_T$ such that  $\paut(P)\subset \Ga_\chi.$
For this, recall  that, for every  $g\in P\subset P_1,$ there exists
$x\in \overline N$ such that $\ga=\paut(g)$ acts as $\Ad(x)$
on $ \overline N$ (Proposition~\ref{Pro-ProjKernNilman}).
 For every unitary character $\chi$ of $\overline N,$ 
we have 
$$
\chi(\vfi(\ga)(y)) = \chi (x yx^{-1}) = \chi(y) \tout y\in \overline N.
$$
Thus, $\paut(P)$ fixes every unitary character  of $\overline N.$

Observe that $\overline N$ is non-trivial, since $\pi\neq 1_N.$
Choose a  non-trivial unitary character  of  $\overline N$ which is constant
on the cosets of $\overline \Lambda$
and denote again by $\chi$ its lift  to $N.$ Then  
 $\chi\in \widehat{T}\setminus\{1_T\}$ and $\chi$ is fixed by $\paut(P)$. 
 $\bsq$

\end{proof}
\begin{remark}
\label{Rem2}
With Remark~\ref{Rem1}, we see that a rough estimate for the integer  $k$ appearing 
in the statement of Proposition~\ref{Pro-FixAut} is 
$$k\leq \frac{1}{4}\left(\dim \left({\Aut} (N)\ltimes N\right) +1\right)^2 +1\leq  \frac{1}{4}((\dim (N))^3 +1)^2 +1.
$$
\end{remark}

\begin{example}
\label{Exa-Heisenberg}
Let $N=H_{2n+1}(\RRR)$ be the $(2n+1)$-dimensional Heisenberg group (over $\RRR$) 
and let $\La$ be  a lattice in $N.$  Then  $\Autnil$ contains a subgroup
of finite index $\Ga$ consisting of automorphisms which fix  
every infinite dimensional representation  $\pi\in \hN$ (see \cite{Folland}).
Let $H$ be a countable subgroup of  $\Affnil$. Assume that
the action of  $H$ on $\nil$ does not 
have a spectral gap.  It follows
from Proposition~\ref{Pro-FixAut} that there is a  subgroup $H_1$ of 
finite index in $H,$  such that the action of $\paut(H_1)$ on $T$
does not have a spectral gap. Therefore, using Theorem~\ref{Theo3},
the action of $H_1$  and hence of the action of $H$ on $T$ does not have a spectral gap.
This result generalizes Theorem~3 in \cite{BeHe} to groups 
of affine transformations of Heisenberg nilmanifolds.
\end{example}

\section{Proof of Theorem~\ref{Theo1}: completion of the proof}
\label{S9}
We are now in position to give the proof of Theorem~\ref{Theo1}.
In view of Theorem~\ref{Theo3}, we only need  show that 
(ii) implies (i).

Let $H$  be a countable subgroup of  $\Affnil$.  Assume, by contraposition,
 that   the action of $H$ on $\nil$ does not have a   spectral gap.
We have to prove that  the action of $H$ on $T$  does not have a spectral gap. 

Set $\Ga=\paut(H).$   By Theorem~\ref{Theo3},  it suffices to prove that  the action  on $T$ of some  subgroup  of  finite index in  $\Ga$  does not have  a spectral gap.
 Let $U^\H$   be  the representation of $\Affnil$  on  the orthogonal complement $\H$ 
 of  $L^2(T)$ in $L^2(\nil)$  and    $\Utor$  the representation   on  $L^2_0(T).$
  Our theorem will be proved if we can show the following

 \medskip
\noindent
\textbf{Claim:}  Let $\mu$ be an  aperiodic measure on $H$.  Assume that  $\Vert U^\H(\mu)\Vert =1$.
 Then there exists a subgroup $\Delta$ of finite index in $\Ga$ and an aperiodic
 probability measure  $\nu$ on $\Delta$ such that
  $\Vert \Utor(\paut(\nu))\Vert=1.$   
  \medskip

To prove this claim, we proceed by induction 
on the dimension  of the Zariski closure 
$\ZC (\Ga)$ of $\Ga $ in $\Aut(N).$

If $\dim  \ZC (\Ga)=0,$ then $\Ga$ is finite and there is nothing to prove.

Assume that $\dim  \ZC (\Ga)\geq 1$ and that the claim above is proved for every 
 countable subgroup of  $H_1$ of $\Affnil$ for which  $\dim  \ZC (\paut(H_1)) <\dim  \ZC (\Ga).$

   Recall from Sections~\ref{S7} and  \ref{S8} that, as $\Ga\ltimes N$-representation,
  $U^\H$ is equivalent to a direct sum
$$\bigoplus_{i\in I}  \ind_{\Ga_i\ltimes N}^{\Ga\ltimes N} V_i,$$
where $\Ga_i$ is the stabilizer in $\Ga$ of  a rational representation $\pi_i\in \hN$ 
and  $V_i$ is a unitary representation  of $\Ga_i\ltimes N$.

Let $I_{\rm fin} \subset I$ be the set of all $i\in I$ such that 
$\Ga_i$ has finite index in $\Ga$ and set $I_{\infty}=I \setminus I_{\rm fin}.$
Let 
$$
U_{\rm fin} =  \bigoplus_{i\in I_{\rm fin} } \ind_{\Ga_i\ltimes N}^{\Ga\ltimes N} V_i \qquad\text{and}
\qquad
U_\infty= \bigoplus_{i\in I_{\infty} } \ind_{\Ga_i\ltimes N}^{\Ga\ltimes N} V_i
$$
and denote by  $\H_{\rm fin}$ and  $\H_{\infty}$  the corresponding subspaces of $\H$ 
defined respectively by $U_{\rm fin}$ and  $U_{\infty}.$
Since  $\Vert U^\H(\mu)\Vert =1,$  two cases can occur.

\noindent
$\bullet$ \emph{First case:}  we have $\Vert U_\infty (\mu)\Vert=1.$
By Herz's majoration principle,  we have
$$
\left\Vert\left(\ind_{\Ga_i\ltimes N}^{\Ga\ltimes N} V_i\right) (\mu) \right\Vert \leq 
\left\Vert \la_{(\Ga\ltimes N)/(\Ga_i\ltimes N)}(\mu)\right\Vert
$$
 for every $i\in I_{\rm fin}.$ Since  $\la_{(\Ga\ltimes N)/(\Ga_i\ltimes N)}= \la_{\Ga/\Ga_i} \circ \paut,$
 it follows that
$$
\left\Vert \bigoplus_{i\in I_{\infty }}\la_{\Ga/\Ga_i}(\paut(\mu))\right\Vert=1.
$$
Let $\eps>0.$ We can choose  $i\in I_{\infty} $ such that 
$$
\Vert \la_{\Ga/\Ga_i}(\paut(\mu))\Vert \geq 1-\eps \leqno{(6)}.
$$

We claim that
$\dim \ZC (\Ga_i)< \dim  \ZC (\Ga).$ Indeed,  otherwise
$\ZC (\Ga_i)$ and   $\ZC (\Ga)$ would have the same connected component $C^0,$ since  $\ZC (\Ga_i) \subset  \ZC (\Ga).$
 As the stabilizer of   $\pi_i$ in $\Aut(N)$
 is Zariski closed (Proposition~\ref {Prop-StabAlg}),  $ C^0 $  would  stabilize $\pi_i.$
 Therefore, $\Ga\cap C^0$ would be contained in $\Ga_i.$ But $\Ga\cap C^0$ 
 has finite index in $\Ga.$ Hence, $\Ga_i$ would have a finite index in $\Ga$
 and this would be a contradiction,   since  $i\in I_{\infty}.$

Let $\mu_i$ be a probability measure  with support equal to  $(\Ga_i\ltimes N)\cap H.$
Then  $(\mu_i+\mu)/2$ is an aperiodic probability measure on $H.$
  Since    $\Vert  U^\H(\mu)\Vert=1,$ we  also have $\Vert  U^\H ((\mu_i+\mu)/2)\Vert=1$.
Therefore,  $\Vert  U^\H(\mu_i)\Vert=1.$ 
Since  $\dim \ZC (\Ga_i)< \dim  \ZC (\Ga),$  it follows from the induction hypothesis
that  $ \Vert \Utor (\mu_i)\Vert =1.$ Then, by Theorem~\ref{Theo3}, we also have
 $ \Vert \Utor (\paut(\mu_i))\Vert =1.$

On the other hand, recall  from $(4)$ that, replacing $\Ga$ by $\Ga_i,$ 
the $\Ga_i$-representation $\Utor$ decomposes into a direct sum 
$$
\Utor \cong \bigoplus_{\chi\in S} \la_{\Ga_i/ (\Ga_\chi \cap \Ga_i )}.
$$
As a consequence, we have
$$
\left\Vert \bigoplus_{\chi\in S} ( \la_{\Ga_i/ (\Ga_\chi \cap \Ga_i) })(\paut(\mu_i))\right\Vert=1. 
$$
Observe that $\paut(\mu_i)$ is an aperiodic probability measure on $\Ga_i$
(in fact, the support of $\paut(\mu_i)$ is $\Ga_i$). It follows that  the $\Ga_i$-representation
$ \bigoplus_{\chi\in S}  \la_{\Ga_i/ (\Ga_\chi \cap \Ga_i) }$ weakly contains
the trivial representation $1_{\Ga_i}$.
Since  
$$
\ind_{\Ga_i}^\Ga 1_{\Ga_i}= \la_{\Ga/ \Ga_i}\qquad\text{and}\qquad  
\ind_{\Ga_i}^\Ga \la_{\Ga_i/ (\Ga_\chi \cap \Ga_i) }= \la_{\Ga/ (\Ga_\chi \cap \Ga_i )}
$$
it follows, by continuity of induction (see Proposition~{F.3.5} in \cite{BHV}), that 
the $\Ga$-representation  $ \bigoplus_{\chi\in S}  \la_{\Ga/ (\Ga_\chi \cap \Ga_i) }$ 
weakly contains  $\la_{\Ga/ \Ga_i}.$  As a consequence, we have
$$
\Vert  \la_{\Ga/ \Ga_i}(\paut(\mu))\Vert\leq \left\Vert  \bigoplus_{\chi\in S} 
(\la_{\Ga/ (\Ga_\chi \cap \Ga_i )})(\paut(\mu))\right\Vert.
$$
Observe that, by Herz's majoration principle again, we have
$$\Vert  \la_{\Ga/ (\Ga_\chi \cap \Ga_i } (\paut(\mu))\Vert \leq  \Vert  \la_{\Ga/ \Ga_\chi  }(\paut(\mu))\Vert. $$
Hence
\begin{align*}
\Vert  \la_{\Ga/ \Ga_i} (\paut(\mu))\Vert &\leq \left\Vert   \bigoplus_{\chi\in S} 
\la_{\Ga/ \Ga_\chi } (\paut(\mu)) \right\Vert\\
&=\Vert \Utor(\paut(\mu))\Vert.
\end{align*}
Using Inequality  $(6),$ it follows that 
$$
\Vert \Utor(\paut(\mu))\Vert \geq 1-\eps.
$$
Since this is true for every $\eps>0,$ we obtain that   $\Vert \Utor(\paut(\mu))\Vert =1.$

\medskip
\noindent
$\bullet$ \emph{Second case:}  we have $\Vert U_{\rm fin}(\mu) \Vert=1.$
By the Noetherian  property of the Zariski topology on $\Aut(N)$, we can find
finitely many indices $i_1, \dots, i_r$ in $I_{\rm fin}$ such that 
$$
\ZC(\Ga_{i_1})\cap\cdots \cap \ZC(\Ga_{i_r}) =\bigcap_{i\in I_{\rm fin}}  \ZC(\Ga_{i}).
$$
 Since stabilizers  of irreducible representations of $N$ are algebraic (Proposition~\ref{Prop-StabAlg}),
  the  subgroup $\Delta:=\Ga_{i_1}\cap \cdots\cap \Ga_{i_r}$ stabilizes   $\pi_i$ for every  $i\in I_{\rm fin}.$
 Moreover, $\Delta$ has finite index in $\Ga,$  since every $\Ga_i$ has finite index in $\Ga.$
 
 From Sections~\ref{S7} and ~\ref{S8}, we have a decomposition of 
 $\H_{\rm fin}$ into $\Delta\ltimes N$-invariant subspaces
 $$
 \H_{\rm fin} =  \bigoplus_{i\in I_{\rm fin} } \H_i,
 $$
 where $\H_i$ is the isotypical component corresponding to $\pi_i$ under 
 the action of $N.$
 Let $\nu$ be a probability  measure  with support equal to ${(\Delta\ltimes N)\cap H.}$ 
 Considering as above
the aperiodic measure  $(\mu+\nu)/2 $ on $H$, we  have 
 $\Vert U_{\rm fin} (\nu))\Vert =1,$ since   $\Vert U_{\rm fin}(\mu) \Vert=1.$ 
 
 On the other hand, by  Proposition~\ref{Pro-FixAut},
there exists an integer $k\geq 1 ,$ which is independent of $i,$ 
such that 
$$
\Vert U_i(\nu))\Vert \leq  \Vert \Utor(\paut(\nu))\Vert^{1/2k} \tout  i\in I_{\rm fin}
$$
where  $U_i$  is the representation of $\Delta \ltimes N$ on  $\H_i$.
As a consequence, we have
$$
\Vert U_{\rm fin} (\nu))\Vert \leq  \Vert \Utor(\paut(\nu))\Vert^{1/2k}
$$
and it follows that $\Vert \Utor(\paut(\nu))\Vert=1.$
Since the support of  $\paut(\nu)$ is the subgroup $\Delta$ 
of finite index  in $\Ga,$ this completes the proof of Theorem~\ref{Theo1}. $\bsq$

\begin{remark}
\label{Rem-Quantitatif}
The proof of Theorem~\ref{Theo1} we gave above
is not effective:  it does not give, for a probability measure
$\mu$ on $\Aut(\nil),$ a bound for the norm of 
$\mu$ under $U^\H$  in terms  of  the norm of $\mu$   
under $\Utor$  and/or  other ''known" representations
of the group generated by $\mu,$ such as  the regular representation. 
In the following example, such an explicit bound is given.
The crucial tool  we use  is Mackey's tensor product theorem
This approach succeeds here  because of the special features
of the example  and  we could not use it to get explicit bounds 
  in the most general case. 
  \end{remark}

 \begin{example}
 \label{Exo-Libre}

 Let ${\frak n}={\frak n}_{3,2} $ be  the free 2-step nilpotent Lie algebra on $3$ generators and
 let $N=N_{3,2}$ be  the corresponding connected and simply-connected nilpotent Lie group.
 As is well-known, ${\frak n}$   is a   6-dimensional Lie algebra 
 which can be realized as follows. 
 Set $V_1=V_2= \RRR^3$ and define a Lie bracket  on 
 the vector space ${\frak n}=V_1 \oplus V_2$ by
 $$
[(X_1, Y_1) , (X_2,Y_2)] = (0, 2 (X_1\wedge X_2)) \tous X_1, X_2, Y_1, Y_2\in \RRR^3,
$$
 where $X_1\wedge X_2$ denotes the usual cross-product on $\RRR^3.$
 (The factor 2 appears here just for computational ease.)
 The centre of $\frak n$ is $V_2$  and the 
 Lie group 
 $N$ is $V_1 \oplus V_2$ with the product
 $$
 (x_1,y_1)(x_2,y_2)= (x_1+x_2, y_1+y_2+x_1\wedge x_2) \tous x_1, x_2, y_1, y_2\in \RRR^3,
 $$
 so that the exponential mapping $\exp: {\frak n}\to N$ is the identity.
  
  Observe that, for a matrix $A\in GL_3(\RRR),$ we have 
 $$
 A(X\wedge Y)= (\det A) (A^t)^{-1}(X\wedge Y) \tous X,Y\in \RRR^3.
 $$
 The automorphism group $\Aut (N)$ of $N$ is the subgroup  of $GL_6(\RRR)$
 of matrices $g_{A,B}$ of the form
\[
g_{A,B}=\left(
\begin{array}{cc}
 A& 0\\
 B&  (\det A) (A^t)^{-1}
\end{array}\
\right)
\]
with $A\in GL_3(\RRR)$ and $B\in M_3(\RRR),$
so  that $\Aut (N)$ is isomorphic to the semi-direct product $GL_3(\RRR)\ltimes M_3(\RRR)$
for the action of $GL_3(\RRR)$ by left multiplication on the  vector space $M_3(\RRR)$ of $3\times 3$-real matrices.

 We will  identify $\frak n$ with $\frak n^*$ by means of the
 standard scalar product  $(X,Y)\mapsto \langle X \vert Y\rangle$
 on $\RRR^6.$
  For $(x,y)$ and  $(X_0,Y_0) $ in $V_1 \oplus V_2$, we compute
  that $ {\rm Ad}^* (x,y)(X_0,Y_0)= (X_0 + x\wedge Y_0, Y_0 ).$
 It follows that the coadjoint orbit of $(X_0,0)$ is $\{(X_0, 0)\}$ and, for $Y_0\neq 0,$ 
 we have
 \begin{align*}
  {\rm Ad}^* (N)(X_0,Y_0)&=\left\{(X_0 +x\wedge Y_0, Y_0 ) \ :\ x\in \RRR^3\right\}\\
 &=\left\{(X_0 +Y, Y_0 ) \ :\ Y\in (\RRR Y_0)^\perp\right\}\\
 &=\left\{(\la_0 Y_0 +Y, Y_0 ) \ :\ Y\in (\RRR Y_0)^\perp\right\}.\\
  \end{align*} 
  for $\la_0= \langle X_0\vert Y_0\rangle/\Vert Y_0\Vert^2.$
 The orbits which are not reduced to  singletons 
  are  therefore the two-dimensional affine planes
 $$
 {\cal O}_{\la_0, Y_0} = \left\{(\la_0 Y_0 +Y, Y_0 ) \ :\ Y\in (\RRR Y_0)^\perp\right\}, 
 $$
 parametrized by $(\la_0, Y_0)\in \RRR\times (\RRR^3\setminus \{0\}).$

The subgroup $\Lambda= \ZZ^3\oplus \ZZ^3$ is a lattice in $N.$
The group $\Aut(\nil)$   is the  subgroup of $\Aut(N)$ of automorphisms $g_{A,B}$ as above given by matrices 
$A\in GL_3(\ZZ)$ and $B\in M_3(\ZZ).$

Fix $(\la_0, Y_0)\in \RRR\times (\RRR^3\setminus \{0\}).$
The irreducible  unitary representation $\pi_{\la_0, Y_0}$
of $N$ corresponding to the coadjoint orbit ${\cal O}_{\la_0, Y_0}$
appears in the decomposition of $L^2(\nil)$ into $N$-isotypical
components if and only if 
${\cal O}_{\la_0, Y_0}\cap  (\ZZ^3\oplus \ZZ^3)\neq \emptyset.$
This is the case if and only if $Y_0\in \ZZ^3\setminus \{0\}$
and $\la_0 \in \Vert Y_0\Vert^{-2} \Delta_{Y_0},$ where $\Delta_{Y_0}$ is the 
subgroup of $\ZZ$ consisting of the integers $m$  for which 
$m Y_0\in (\RRR Y_0)^\perp +  \Vert Y_0\Vert^2 \ZZ^3.$

Let $\Ga$ be a subgroup of $\Aut(\nil).$ For simplicity, 
we assume that  $\Ga$ consists only of automorphisms $g_{A,0}$
with $A\in SL_3(\ZZ).$  
We identify $\Ga$ with  a subgroup of $SL_3(\ZZ).$
For $A\in SL_3(\ZZ),$ we have 
$$
A({\cal O}_{\la_0, Y_0})= {\cal O}_{\beta_0, (A^t)^{-1}(Y_0)}\qquad\text{for}\qquad
\beta_0={ \la_0\Vert Y_0\Vert^2}/{\Vert (A^t)^{-1}(Y_0)\Vert^2}.
$$
The stabilizer  $\Ga_{\la_0,Y_0}$ of ${\cal O}_{\la_0, Y_0}$ 
(which is the stabilizer of $\pi_{\la_0, Y_0}$)  in $\Ga$
is therefore
\begin{align*}
\Ga_{\la_0,Y_0}&= \{A\in \Ga \ : \ A^{t} Y_0= Y_0 \},
\end{align*}
and is isomorphic to a subgroup of the semi-direct product $SL_2(\ZZ)\ltimes \ZZ^2.$

Let $\H_{\la_0,Y_0}$ be the isotypical component of 
$L^2(\nil)$ associated to  $\pi_{\la_0, Y_0}$
and $U_{\la_0, Y_0}$ the corresponding representation  of 
$\Ga$ (see Section \ref{S7}); we know that 
$U_{\la_0, Y_0}$
is equivalent  to $ \ind_{\Ga_{\la_0,Y_0}}^{\Ga} V_{\la_0,Y_0}$
for a representation $V_{\la_0,Y_0}$ of $\Ga_{\la_0,Y_0}$ which is 
strongly $L^p$ modulo its projective kernel $P_{\la_0, Y_0}$
for some real number $p\geq 1.$

The projective kernel $P_{\la_0, Y_0}$ of $V_{\la_0,Y_0}$ coincides with
the subgroup of $\Ga$ of all automorphisms
which fixes every point $(X,Y)\in {\cal O}_{\la_0,Y_0};$
hence,  $P_{\la_0, Y_0} =\{I\}$ if $\la_0=0$ and 
$$
P_{\la_0, Y_0}=\{A\in \Ga \ : \ A^{t} Y_0= Y_0 \quad \text{and}\quad AY=Y \quad\text{ for all}\quad Y\in (\RRR Y_0)^\perp\} 
$$
 if $\la_0\neq 0.$

Every $\pi_{\la_0, Y_0}$ factorizes to a representation of 
a quotient of $N$ of dimension $3$ or $4,$
which is isomorphic to the Heisenberg group $H_3$
or to the direct product  $H_3\oplus \RRR.$ 
It follows that the representation $V_{\la_0,Y_0}$ of $\Ga_{\la_0,Y_0}$  is 
strongly $L^{6+\eps}$ modulo $P_{\la_0, Y_0}$ for every $\eps>0$
(see \cite{BeHe} and \cite{HoMo}).
 
 Set $\Ga_0=\Ga_{\la_0,Y_0}, V=V_{\la_0,Y_0},$ and $U=U_{\la_0,Y_0}.$
 We claim that $U^{\otimes 4}$ is weakly contained in the regular representation $\la_\Ga$ of $\Ga$
on $\ell^2(\Ga).$

  Indeed, by Mackey's tensor product theorem, 
 $U^{\otimes 4}$ is weakly equivalent to the direct sum
 $$
\bigoplus_{\ga_1,\ga_2,\ga_3\in\Ga}  \ind_{\Ga_0\cap \Ga_{0}^{\ga_1} \cap \Ga_{0}^{\ga_2} \cap \Ga_{0}^{\ga_3} }^{\Ga}
 \left(V\otimes V^{\ga_1}\otimes V^{\ga_2}\otimes V^{\ga_3}\right),
  $$
  where $V\otimes V^{\ga_1}\otimes V^{\ga_2}\otimes V^{\ga_3}$  is the tensor product
  of the restrictions of   $V, V^{\ga_1},V^{\ga_2}$ and $V^{\ga_3}$
  to $\Ga_0\cap \Ga_{0}^{\ga_1} \cap \Ga_{0}^{\ga_2} \cap \Ga_{0}^{\ga_3}.$ 
   Fix $\ga_1,\ga_2, \ga_3\in\Ga.$  Observe that 
   $\Ga_0\cap  \Ga_{0}^{\ga_1} \cap \Ga_{0}^{\ga_2} \cap \Ga_{0}^{\ga_3}$
   is the subgroup of elements $\ga\in \Ga$ 
   such that $\ga^t$ fixes  $Y_0, \ga_1^t(Y_0), \ga_2^t(Y_0)$ and $\ga_3^t(Y_0).$
   Set 
   $$U_{\ga_1,\ga_2\ga_3}=\ind_{\Ga_0\cap \Ga_{0}^{\ga_1} \cap \Ga_{0}^{\ga_2} \cap \Ga_{0}^{\ga_3} }^{\Ga}
 \left(V\otimes V^{\ga_1}\otimes V^{\ga_2}\otimes V^{\ga_3}\right).
 $$ 
 Two cases can occur.
   
\medskip \noindent
$\bullet$ \emph{First case:} There exists some $i\in \{1,2,3\}$ such that 
$\ga_i^t(Y_0)$ is not a multiple of $Y_0.$  
Then every element $\Ga_0\cap  \Ga_{0}^{\ga_1} \cap \Ga_{0}^{\ga_2} \cap \Ga_{0}^{\ga_3}$
fixes pointwise a plane in $\RRR^3;$
it follows that  $\Ga_0\cap  \Ga_{0}^{\ga_1} \cap \Ga_{0}^{\ga_2} \cap \Ga_{0}^{\ga_3}$ is abelian and hence amenable. Therefore
$U_{\ga_1,\ga_2\ga_3}$ is weakly contained in  $\la_\Ga.$

\medskip \noindent
$\bullet$ \emph{Second case:}
Every  $\ga_i^t(Y_0)$ is  a multiple of $Y_0,$
that is, every $\ga_i$ belongs to the 
subgroup  $ H= \left\{ \ga\in \Ga\ : \ \ga^t(Y_0) \in \{\pm Y_0\}\right\}.$
 Observe that $\Ga_0$ is a subgroup of $H$ of index at most 2.
 It can be checked that the  subgroup
 $P=P_{\la_0, Y_0},$ which is normal  in $\Ga_0,$ is normal
 in $H.$ It follows that the restriction 
 of  $ V^{\ga_i}$ to  $\Ga_0\cap \Ga_{0}^{\ga_1} \cap \Ga_{0}^{\ga_2} \cap \Ga_{0}^{\ga_3}$ 
 is strongly  $L^{6+\eps} $ modulo $P$ for every  $i\in \{1,2,3\}.$ 
Hence,  $V\otimes V^{\ga_1}\otimes V^{\ga_2}\otimes V^{\ga_3}$
 is strongly $L^{2} $ modulo $P$ and hence contained  in a multiple
 of  $\ind _P^{\Ga_0\cap \Ga_{0}^{\ga_1} \cap \Ga_{0}^{\ga_2} \cap \Ga_{0}^{\ga_3}}\la.$
Since $P$ is amenable, it follows that 
 $U_{\ga_1,\ga_2,\ga_3}$ is  weakly contained in  $\la_\Ga.$ 
 As a consequence,  we see that  $U^{\otimes 4}$ is weakly contained in  $\la_\Ga.$
 
 Let $\mu$ be a probability  measure on $\Ga.$
 It follows from what we have seen that 
$$
\Vert U^\H(\mu)\Vert\leq \Vert \lambda_\Ga(\mu)\Vert^{1/4},
$$
where  $U^\H$ is the Koopman representation of $\Ga$
on $\H = L^2(T)^\perp.$ As a consequence, we have
$$
\Vert U^0(\mu)\Vert\leq \max\{\Vert \lambda_\Ga(\mu)\Vert^{1/4},\Vert \Utor(\mu)\Vert \},
$$
where  $U^0$ and $\Utor$ are the Koopman representations of $\Ga$
on  $L^2_0(\nil)$ and $ L^2_0(T).$
The same estimate was established  in \cite[Corollary~3]{BeHe}
in the case where $N$ is  the  Heisenberg group $H_3.$
 \end{example}

\section{Proof of Theorem~\ref{Theo2}}
\label{S-ErgMixing}
Let $H$ be a  subgroup  of  $\Affnil.$
The following elementary proposition
shows that ergodicity of $H$ on $T$ is inherited by every subgroup
of finite index in $H.$
\begin{proposition}
\label{Pro-FiniteIndexErg}
Let $H$ be a subgroup of $\Aff(T)$  and $H_1$ a subgroup of finite index in $H.$
 Assume  that 
$L^2_0(T)$ contains a non-zero $H_1$-invariant function. 
Then $L^2_0(T)$ contains a non-zero $H$-invariant function. 
\end{proposition}
\begin{proof}
By standard arguments involving Fourier series, there exists
a unitary character $\chi$ in $\widehat{T}\setminus\{1_T\}$ 
with a finite orbit under $\paut(H_1)$ and such that 
$H_2:=H_1\cap \paut^{-1}(\Ga_\chi)$ fixes $\chi,$ 
where $\Ga_\chi$ is the stabilizer of $\chi$ in $\Aut(T).$ 
Then $H_2$
has finite index in $H$ and 
$$\sum_{s\in H/ H_2}  \Utor(s)\chi$$ is a 
non-zero $H$-invariant function in $L^2_0(T).\bsq$
\end{proof}

\n
\textbf{Proof of (i) in Theorem~\ref{Theo2}}
\medskip

 As is well-known,
the action of  a group $H$ on a probability space $(X,\nu)$ is weakly mixing 
if and only if the diagonal action of $H$ on $(X\times X,\nu\otimes \nu)$ 
is ergodic. Since $T\times T$ is the maximal factor torus of $(\nil)\times(\nil),$
we only have to prove the statement about ergodicity.

So, let $H$ be a (not necessarily countable) subgroup of $\Affnil$ acting ergodically on $T.$
We have to prove that $H$ acts ergodically on $\nil.$
We can assume that $N$ is not abelian, otherwise there is nothing to prove.

Set $\Ga=\paut(H).$ Recall  from Sections~\ref{S7} and  \ref{S8} that 
we have  orthogonal decompositions  
into $\Ga\ltimes N$ -invariant subspaces $L^2(\nil)=L^2(T)\oplus \H$ 
and 
$$
 \H= \bigoplus_{i} \H_{\Sigma_i},
 $$
 such that   the representation $U_i$  of $\Ga\ltimes N$  on $\H_{\Sigma_i}$ 
is equivalent to an induced representation  $ \ind_{\Ga_{\pi_i}\ltimes N}^{\Ga\ltimes N} V_{_i},$
where $\Ga_{\pi_i}$ is the stabilizer in $\Ga$ of some  $\pi_i\in \Sigma_i.$ 
In view of the previous proposition, it suffices to prove the following

\medskip
\noindent
\textbf{Claim:}  Assume that, for some $i,$
the   subspace $\H_{\Sigma_i}$ contains a non-zero $H$-invariant function.
 Then   $L^2_0(T)$ contains a non-zero $H_1$-invariant function
for some subgroup $H_1$ of finite index in $H.$
\medskip

To show this,  set $\pi=\pi_i$, $\Sigma_\pi=\Sigma_i,$ $U_\pi=U_i,$
and $V_\pi=V_i.$  Let    $S$ be a set of representatives
for the cosets in 
$$\Ga/\Ga_{\pi}\cong (\Ga\ltimes N)/ (\Ga_\pi\ltimes N)$$
 with $e\in S.$ 
Then, by the definition of an induced representation,  $\H_{\Sigma_\pi}$ is an orthogonal sum 
$$ \H_{\Sigma_\pi}= \bigoplus_{s\in S}\K^s,$$
where  $\K$  carries the $\Ga_{\pi}\ltimes N$-representation $V_{\pi}$ 
 and where  $\K^s=U_\pi(s)\K.$
 It follows from this that there exists a non-zero  function in $\K$
which is invariant under $H\cap (\Ga_{\pi}\ltimes N)$
 and that $\Ga_{\pi}$ has finite index in $\Ga.$ 
 
 Upon replacing $H$ by the subgroup 
 of finite index  $H\cap (\Ga_{\pi}\ltimes N),$
 we can assume that $H$ is contained in  $\Ga_{\pi}\ltimes N$.
 
Let $L_\pi$ be the connected component of $\Ker (\pi)$
and $\overline N=N/L_\pi.$  Observe that $\overline N$ is not abelian,
since $\pi$ is not a unitary character of $N.$
As seen in Section~\ref{S5},  
the action of $\Ga_{\pi}\ltimes N$ 
on $\H_\pi$ factorizes through 
the quotient nilmanifold $\overline \La\bs \overline N.$
Hence, we can assume that $L_\pi$ is trivial.

 By the proof of Proposition~\ref{Pro-FixAut}, 
 there exists a real number $p\geq 1$ such that the representation $V_\pi$ of 
 $\Ga_\pi\ltimes N$ is strongly $L^p$  modulo  $\Delta,$
 where  $\Delta$ is the normal subgroup
 $$\Delta= \{(\Ad(x), x^{-1} z)\ :  x\in  \La, z\in Z(N)  \}.$$
  We claim that $H\cap \Delta$ has finite index in $H.$

 Indeed, let $R=\overline{H\Delta}$ be the closure of $H\Delta$ in  $\Ga_\pi\ltimes N$.
 Then  the restriction of $V_\pi$ to 
 $R$ is strongly $L^p$  modulo  $\Delta.$
 
 Observe that  
 $(\Ad(x), x^{-1} z)\in \Delta$ acts as
 multiplication with $\la_\pi(z)$ on $\H_\pi,$
 where $\la_\pi$ is the central character  of $\pi.$
 Let $\xi$ a non-zero $V_\pi(H)$-invariant function in $\K.$
 The function  $x\mapsto \vert  \langle V_\pi(x)\xi , \xi\rangle\vert$
 is non-zero, belongs to $L^p(R/\Delta),$ and is $R$ invariant.
 It follows that $R/\Delta$ is  a compact group.
 
 Let $R_0$ be the connected component of $R.$ Since 
 $R$ is a Lie group,  $R_0$ is open in $R.$   
 It follows that  $R_0 \Delta/\Delta$ is an open (and hence closed)
 subgroup of   $R/\Delta.$
 Since $R/\Delta$ is compact, we conclude that $R_0\Delta /\Delta \cong R_0/(R_0\cap \Delta)$ 
 is a subgroup of finite index in  $R/\Delta.$
 
 On the other hand, observe that   $R_0\subset N,$ 
 since $R\subset \Ga_\pi\ltimes N$ and since  $\Ga_\pi$  is discrete.
 Observe also that 
 $$R_0\cap \Delta = R_0 \cap Z(N),$$
 since $Z(N)$ is connected (as $N$ is simply connected). It follows that
 $R_0\cap \Delta$ is  a connected
 subgroup of the nilpotent simply connected Lie group $R_0.$
 But $R_0/(R_0\cap \Delta)$ is compact. Hence,
 $R_0/(R_0\cap \Delta)$  is trivial.
 As a consequence, we see that  $R/\Delta$ is finite.
 This shows that $H\cap \Delta$ has finite index in $H.$ 
 Therefore, upon replacing $H$ by  $H\cap \Delta$, we can assume that $H\subset \Delta.$ 
 
The centre $Z(N)$ being a rational
subgroup of $N,$   the subgroup $\overline{\La}=\La Z(N)$ 
of  the   nilpotent Lie group $\overline N= N/ Z(N)$ is a lattice. 
Observe that $\overline N$  is non-trivial, since 
$N$ is non-abelian. 
 The group $\Delta$ acts  trivially on the factor nilmanifold 
 $\overline \La\bs \overline N$ and hence  on the associated
 torus $\overline{T}$. 
 Since  $\overline{T}$ is a $\Delta$-invariant factor torus of $T,$
 it follows that   the action of $H$ on $T$  is not ergodic.$\bsq$

\bigskip\n
\textbf{Proof of (ii) in Theorem~\ref{Theo2}}
\medskip

Let $H$ be a subgroup of $\Autnil$ with a strongly mixing action on $T.$
We have to prove that the action of $H$  on $\nil$ is strongly mixing.

With the notation as in the proof of Part (i) above,
 the Koopman representation 
$U$ of $H$ on $\H$ decomposes as a direct
sum $U\cong \oplus_{i} U_i,$
where $U_i$ equivalent to an induced representation  $ \ind_{H_{\pi_i} }^{H} V_{i}.$
It suffices to prove that, for every  $i,$
the matrix coefficients of $U_i$  belong to $c_0(H).$ 
This will follow if we show that the matrix coefficients of $V_i$ belong
to  $c_0(H_{\pi_i}).$ 

Set $\pi=\pi_i$ and $V_\pi=V_i.$ Let $L_\pi$ be the connected component of $\Ker(\pi)$
and $\overline \La\bs \overline N$ the corresponding $H_\pi$-invariant 
factor nilmanifold. Since $H_\pi$ is contained in $ \Autnil,$  the projective kernel 
$P$ of $V_\pi$  coincides with the kernel  of the homomorphism
$\vfi: H_\pi\to \Aut (\overline \La\bs \overline N)$, by  Proposition~\ref{Pro-ProjKernNilman}.

We claim  that  $P=\Ker(\vfi)$ is finite.
Indeed, otherwise   the matrix coefficients of the Koopman representation
of $H_\pi$ on the maximal factor torus $\overline T$  of $\overline \La\bs \overline N$
would  not belong to $c_0(H_\pi)$ and this would imply that the  action of $H_\pi$
and hence of $H$ on $T$ is not strongly mixing.

Since $P$ is finite, $V_\pi$  is strongly  $L^p$ for some $p\geq 1.$ 
It follows that the matrix coefficients of $V_\pi$  belong 
 to $c_0(H_\pi).$ 
This  finishes the proof of  Theorem~\ref{Theo2}.$\bsq$

\noindent
{\bf Address}

\noindent
\obeylines 
{Bachir Bekka and Yves Guivarc'h
IRMAR, Universit\'e de  Rennes 1, 
Campus Beaulieu, F-35042  Rennes Cedex
 France}

\noindent
E-mail : bachir.bekka@univ-rennes1.fr, yves.guivarch@univ-rennes1.fr

\end{document}